\documentclass[a4paper, 11pt]{amsart}
\input xy   
\xyoption{all}
\swapnumbers
\usepackage{amssymb}  
\usepackage{amsthm}
\usepackage{enumerate}

\numberwithin{equation}{section}

\theoremstyle{plain}
  \begingroup
        \newtheorem{theorem}[equation]{Theorem}
        
        \newtheorem{proposition}[equation]{Proposition}

	    \newtheorem{definition}[equation]{Definition}

        \newtheorem{sinnadaitalica}[equation]{}
\endgroup

\theoremstyle{definition}
  \begingroup
        \newtheorem{remark}[equation]{Remark}

        \newtheorem{sinnadastandard}[equation]{}

\endgroup

\newcommand{\mr}[1]{\buildrel {#1} \over \longrightarrow}

\newcommand{\Mr}[1]{\buildrel {#1} \over \Longrightarrow}

\newcommand{\dr}[1]
          {
           \ar@<4pt>@{-}'+<0pt,-6pt>[d] 
           \ar@<-4pt>@{-}'+<0pt,-6pt>[d]^{#1}
          }

\newcommand{\dd}{\ar@2{-}[d]}

\newcommand{\op}[1]
          {
           \ar@{-}[ld] 
           \ar@{-}[rd] 
           \ar@{}[d]|{#1}  
          }

\newcommand{\cl}[1]
          { 
           \ar@{-}[ur] 
           \ar@{}[u]|{#1} 
           \ar@{-}[ul] 
          }

\newcommand{\ig}[1]
          {
           \ \ \ \ar@{}[d]|{\stackrel{#1}{=}}
          }

\newcommand{\X}{\ell X}

\newcommand{\dcell}[1]
          {
           \ar@<4pt>@{-}'+<0pt,-6pt>[d] 
           \ar@<-4pt>@{-}'+<0pt,-6pt>[d]^{#1}
          }

\newcommand{\dcellb}[1]
          {
           \ar@<5pt>@{-}'+<0pt,-6pt>[d] 
           \ar@<-5pt>@{-}'+<0pt,-6pt>[d]^{#1}
          }

\newcommand{\did}       % identidad (down)%
         {
          \ar@{=}[d]
         }

\newcommand{\pmr}[2]
{
\xymatrix@C=5ex@R=2.4ex
         {
          {} \ar@<1.6ex>[r]^{#1} 
	     \ar@<-1.1ex>[r]^{#2} & {}
         }
}

\newcommand{\pml}[2]
{
\xymatrix@C=5ex@R=2.4ex
         {
            {} 
          & {} \ar@<1.0ex>[l]_{#1} 
	       \ar@<-1.7ex>[l]_{#2}
         }
}

\newcommand{\cellr}[3]
{
\xymatrix@C=7ex@R=2.4ex
         {
          {} \ar@<1.6ex>[r]^{#1} 
          \ar@{}@<-1.3ex>[r]^{\!\! #2 \, \!\Downarrow}
                                         \ar@<-1.1ex>[r]_{#3} & {}
         }
}

\newcommand{\celll}[3]
{
\xymatrix@C=7ex@R=2.4ex
         {
            {} 
          & {} \ar@<1.0ex>[l]^{#1} 
          \ar@{}@<-1.7ex>[l]^{\!\! #2 \, \!\Downarrow}
	                                 \ar@<-1.7ex>[l]_{#3}
         }
}

\newcommand{\cc}{\mathcal}

\newcommand{\mono}{\hookrightarrow}
\newcommand{\mmr}[1]{\buildrel {#1} \over \hookrightarrow}

\usepackage{colordvi}

\newcommand{\cqd}{\hfill$\Box$}

\usepackage{hyperref}
\usepackage[latin1]{inputenc}
\usepackage[english]{babel}
\usepackage{color}
\usepackage{verbatim}
\usepackage{amsmath, amsthm, amsfonts, amssymb}
\usepackage[all]{xy}
 
\usepackage{graphicx}

	\def \Cat{\mathcal{C}}
	
	\def \Eat{\mathcal{E}}
	
	\def \Xat{\mathcal{X}}
	
	\def \Vat{\mathcal{V}}
	\def \Ens{\mathcal{E}ns}

	\def \be{\begin{enumerate}}
	\def \en{\end{enumerate}}

\begin{document}

\title{A Tannakian context for Galois}

\author{Eduardo J. Dubuc and Martin Szyld}

\begin{abstract}
Strong similarities have been long observed between the Galois (Categories Galoisiennes) and the Tannaka (Categories Tannakiennes) theories of representation of groups. In this paper we construct an explicit (neutral) Tannakian context for the Galois theory of atomic topoi, and prove the equivalence between its fundamental theorems. Since the theorem is known for the Galois context, this yields, in particular, a proof of the fundamental (recognition) theorem for a new Tannakian context. This example is different from the additive cases or their generalization, where the theorem is known to hold, and where the unit of the tensor product is always an object of finite presentation, which is not the case in our context.
\end{abstract}

\maketitle

{\bf Introduction.} Strong similarities have been long observed between the Galois (Categories Galoisiennes) and the Tannaka (Categories Tannakiennes) theories of representation of groups. In this paper we construct an explicit (neutral) Tannakian context for the Galois theory of atomic topoi, and prove the equivalence between its fundamental theorems. Since the theorem is known for the Galois context, this yields, in particular, a proof of the fundamental (recognition) theorem for a new Tannakian context. This example is different from the additive cases \cite{JS}, \cite{H}, \cite{C}, or their generalization \cite{SC}, where the theorem is known to hold, and where the unit of the tensor product is always an object of finite presentation (that is, filtered colimits in the tensor category are constructed as in the category of sets), which is not the case in our context.
Very different approaches to relate Tannaka with Galois are developed in \cite{R} and \cite{JAS}, where the existence of Galois closures (disguised in one form or another) is essential, and which cover Galois topoi but not the Joyal-Tierney extension to atomic topoi.

\vspace{1ex}

In this paper by Galois theory we mean Grothendieck's Galois theory of progroups (or prodiscrete localic groups) and Galois topoi \cite{G1}, \cite{G2}, as extended by 
Joyal-Tierney in \cite{JT}. More precisely, the particular case of arbitrary localic groups and pointed atomic topoi. 

For the Galois theory of atomic topoi we follow Dubuc \cite{D1}, where he develops \emph{localic Galois theory} and makes a explicit construction of the localic group of automorphisms $Aut(F)$ of a set-valued functor  $\cc{E} \mr{F} \cc{E}ns$, and of a lifting $\cc{E} \mr{\widetilde{F}} \beta^{Aut(F)}$ into the topos of sets furnished with an action of the localic group (see \ref{Aut(F)}). He proves in an elementary way\footnote{meaning, without recourse to change of base and other sophisticated tools of topos theory over an arbitrary base topos.} 
that when $F$ is the inverse image of a point of an atomic topos, this lifting is an equivalence \mbox{\cite[Theorem 8.3]{D1},} which is Joyal-Tierney's theorem \cite[Theorem 1]{JT}. 

For Tannaka theory we follow Joyal-Street \cite{JS} (for the original sources see the references therein). The construction of the Hopf algebra $End^\vee(T)$ of endomorphisms of a finite dimensional vector space valued functor $T$ can be developed for a $\cc{V}_0$-valued functor, \mbox{$\cc{X} \mr{T} \cc{V}_0 \subset \cc{V}$,} where $\cc{V}$ is any cocomplete monoidal closed category, and $\cc{V}_0$ a (small) full subcategory of objects with duals, see for example \cite{P}, \cite{SC}, \cite{SP}. There is a lifting 
\mbox{$\cc{X} \mr{\widetilde{T}} Cmd_0(End^\vee(T))$} into the category of \mbox{$End^\vee(T)$-comodules} with underlying object in $\cc{V}_0$. For a handy reference and terminology see \mbox{appendix \ref{appendix}.} In \cite{JS}, \cite{SP} it is shown that in the case of vector spaces, if 
$\cc{X}$ is abelian and $F$ is faithful and exact, the lifting is an equivalence (recognition theorem). 

\vspace{1ex}

{\bf Contents.}

Recall that given a regular category $\Cat$ we can consider the category $\cc{R}el(\Cat)$ of relations in $\Cat$. There is a faithful functor (the identity on objects) \mbox{$\Cat \rightarrow \cc{R}el(\Cat)$,} and any regular functor $\cc{C} \mr{F} \cc{D}$ has an extension 
\mbox{$\cc{R}el(\cc{C}) \mr{\cc{R}el(F)} \cc{R}el(\cc{D})$.}

The category $\cc{R}el = \cc{R}el(\cc{E}ns)$ is a full subcategory of the category $Sup$ of sup-lattices, set $\cc{R}el = Sup_0$. This determines the base $\cc{V},\cc{V}_0$ of a Tannaka context. Furthermore, a localic group is the same thing as an idempotent Hopf algebra in the category $Sup$ (see section \ref{background}).

Given any pointed topos with inverse image $\cc{E} \mr{F} \cc{E}ns$ of a Galois context, we associate a (neutral) Tannakian context as follows:
$$
\xymatrix
       {
	        \beta^G \ar[rd]_{} 
	      & \Eat \ar[l]_{\widetilde{F}} \ar[d]^{F} \ar[r] 
	      & Rel(\Eat) \ar[d]_{T} \ar[r]^{\widetilde{T}} 
	      & Cmd_0(H) \ar[dl]^{} 
	     \\
		  & \Ens \ar[r]
		  & \cc{R}el = Sup_0,
		 }
$$
where $G = Aut(F)$, $H = End^\vee(T)$, and $T = \cc{R}el(F)$.

\vspace{1ex}

We prove that \emph{$\widetilde{F}$ is an equivalence if and only if 
$\widetilde{T}$ is so} \mbox{(Theorem \ref{AA}).}  The result is based in two theorems. First, we prove that for any localic group $G$, there is an isomorphism of categories $\cc{R}el(\beta^G) \cong Cmd_0(G)$ 
(Theorem \ref{Comd=Rel}). Second, we prove that the Hopf algebra $End^\vee(T)$ is a locale, and that there is an isomorphism $Aut(F) \cong End^\vee(T)$ (Theorem \ref{G=H}).

In particular, from Theorem \ref{AA} and the fundamental theorem of localic Galois theory (Theorem \ref{fundamentalLGT}), it follows that the \emph{Tannaka recognition theorem holds in the (neutral) Tannaka context associated to a pointed topos if and only if the topos is connected atomic (Theorem \ref{BB}).} These topoi are then a new concrete example where the theorem holds wich is completely different than the other cases in which the Tannaka recognition theorem is known to hold, where the unit of the tensor product is an object of finite presentation. Simultaneously, the non atomic topoi furnish examples where the theorem is false.  

\vspace{1ex}

{\bf Acknowledgements.} The first author thanks Andr\'e Joyal for many stimulating and helpful discussions on the subject of this paper.

%\tableofcontents

\section{Background, terminology and notation} \label{background}

In this section we recall some facts on sup-lattices, locales and
monoidal categories, and in this way we fix notation and terminology.

We will consider the monoidal category $Sup$ of \emph{sup-lattices},
whose objects are posets $S$ with arbitrary suprema $\bigvee$ (hence
finite infima $\wedge$, $0$ and $1$) and 
whose arrows are the suprema-preserving-maps. We call these arrows
\emph{linear maps}. We will write $S$ also
for the underlying set of the lattice.  
The \emph{tensor product} of two sup-lattices $S$ and $T$ is the codomain
 of the universal bilinear map $S \times T \mr{} S \otimes T$. Given 
 $(s,\,t) \in S \times T$, we denote the corresponding element in 
 $S \otimes T$ 
by $s \otimes t$. The unit for $\otimes$ is the sup-lattice $2=\{0 \leq 1\}$. The linear map $S \otimes T \stackrel{\psi}{\rightarrow} T \otimes S$ sending 
$s \otimes t \mapsto t \otimes s$ is a symmetry. 
Recall that, as in any monoidal category, a \emph{duality} between two sup-lattices $T$ and $S$ is a pair of arrows $2 \stackrel{\eta}{\rightarrow} T \otimes S$, 
$S \otimes T \stackrel{\varepsilon}{\rightarrow} 2$ satisfying the usual triangular equations (see \ref{triangular}). We say that $T$ is \emph{right dual} to $S$ and that $S$ is \emph{left dual} to $T$, and denote $T = S^\wedge$, $S = T^\vee$.

 There is a free sup-lattice functor $\Ens \mr{\ell} Sup$. Given  $X \in \cc{E}ns$, $\ell X$ is the power set of $X$, and for    
\mbox{$X \stackrel{f}{\rightarrow} Y$}, $\ell f = f$ is the direct image. This functor extends to a functor $\cc{R}el \mr{\ell} Sup$, defined on the category $\cc{R}el$ of sets with relations as morphisms. A linear map $\ell X \to \ell Y$ is the ``same thing'' as a relation $R \subset X \times Y$. In this way $\cc{R}el$ can be identified with a full subcategory $\cc{R}el \mmr{\ell} Sup$. We define $Sup_0$ as the full subcategory of $Sup$ of objects of
the form $\ell X$. Thus, abusing notation, $\cc{R}el = Sup_0 \subset Sup$ (``$=$'' here is actually an isomorphism of categories).
Recall that $\cc{R}el$ is a monoidal category with tensor product given by the cartesian product of sets (which is not a cartesian product in $\cc{R}el$). It is immediate to check that $\ell X \otimes \ell Y = \ell(X \times Y)$ in a natural way.
\begin{sinnadaitalica} \label{tensoriso}
The functor $\cc{R}el \mmr{\ell} Sup$ is a tensor functor, and the identification $\cc{R}el = Sup_0$ is an isomorphism of monoidal categories.
\end{sinnadaitalica}
The arrows
$2 \stackrel{\eta}{\rightarrow} \ell X \otimes \ell X$, \mbox{$\ell X
\otimes \ell X \stackrel{\varepsilon}{\rightarrow} 2$,} defined on the generators as $\eta(1) = \bigvee_x x \otimes x$ and 
\mbox{$\varepsilon (x\otimes y) = \delta_{x=y}$} determine a duality, and in this way the objects
of the form  $\ell X$ have both duals and furthermore they are self-dual, 
\mbox{$(\ell X)^\wedge = (\ell X)^\vee = \ell X$.} Under the isomorphism 
$\cc{R}el = Sup_0$, $\varepsilon$ and $\eta$ both correspond to the diagonal relation $\Delta \subset X \times X$. Duals are contravariant functors, if $R \subset X \times Y$ is the relation corresponding to a linear map 
\mbox{$\ell X \to \ell Y$,} then the opposite relation $R^{op} \subset Y \times X$ corresponds to the dual map 
$(\ell Y)^\wedge  \to (\ell X)^\wedge$.
\begin{sinnadaitalica} \label{ellX=X}
We will abuse notation by identifying  $X$, $\ell X$ and $(\ell X)^\wedge$, a function with its graph, and the 
inverse image of a function with the opposite relation.
\end{sinnadaitalica}
As in any monoidal category, an \emph{algebra} (or \emph{monoid})
in $Sup$ is an object $G$ with an associative multiplication $G \otimes G \mr{*} G$ which 
has a unit $u \in G$. If $*$ preserves the
symmetry $\psi$, the algebra is \emph{commutative}. An algebra morphism is a
linear map which preserves $*$ and $u$.

A \emph{locale} is a sup-lattice $G$ where finite infima $\wedge$
distributes over arbitrary suprema $\bigvee$, that is, it is bilinear, and
so induces a multiplication \mbox{$G \otimes G \mr{\wedge} G$.}
A locale morphism is a linear map which preserves \mbox{$\wedge$ and
$1$.} In
this way locales are commutative algebras, and there is a full inclusion
of categories $Loc \subset Alg_{Sup}$ into the category of commutative algebras in $Sup$.
\begin{sinnadaitalica} \label{charlocales}
In \cite{JT} locales are characterized as those commutative algebras such
that $x * x = x$ and $u = 1$.
\end{sinnadaitalica}

A (commutative) \emph{Hopf algebra} in $Sup$ is a group object in $(Alg_{Sup})^{op}$. 
A \emph{localic group} (resp. \emph{monoid}) $G$ is a group (resp. monoid) object in
the category $Sp$ of \emph{localic 
spaces}, which is defined to be the formal dual of the category of locales, $Sp =
Loc^{op}$. Therefore $G$ can be also considered as a Hopf algebra in
$Sup$. The unit and the multiplication of $G$ in $Sp$ are the counit $G \mr{e} 2$ and comultiplication 
$G \mr{w} G \otimes G$ of a coalgebra structure
for $G$ in $Alg_{Sup}$. The inverse is an antipode 
$G \mr{\iota} G$. Morphisms correspond but change direction, and we actually have a contravariant equality of
categories \mbox{$(Id$-$Hopf)^{op} = Loc$-$Group$,} between the category of
idempotent (with $u=1$) Hopf algebras in $Sup$ and the category of localic groups.  

\section{Preliminaries on bijections with values in a locale} \label{sec:general}

As usual we view a relation $\lambda$ between two sets $X$ and $Y$ as
a map (i.e. function of sets) $X \times Y \mr{\lambda}2$. We consider maps $X \times Y \mr{\lambda} G$ with values in an arbitrary sup-lattice $G$, that we will call \emph{$\ell$-relations}. Since $\ell(X \times Y) = \ell X \otimes \ell Y$, it follows that $\ell$-relations are the same thing that linear maps $\ell X \otimes \ell Y \mr{\lambda} G$. 
The results of this section are established in order to be used in the 
next sections, and they are needed only in the case $X=Y$.

\begin{sinnadastandard} \label{diagramadiamante12}
Consider two $\ell$-relations $X \times Y \mr{\lambda} G$, \mbox{$X' \times
Y'\mr{\lambda'}G$,} and two maps $X \mr{f} X'$, $Y \mr{g} Y'$, or, more generally, two spans (which induce relations that we also denote with the same letters), 
$$
\xymatrix@C=3ex@R=1ex
         {
         {} & R \ar[dl]_{p} \ar[dr]^{p'}
         \\
         X & {} & X',
         }
\hspace{3ex}                  
\xymatrix@C=3ex@R=1ex
         {
         {} & S \ar[dl]_{q} \ar[dr]^{q'}
         \\
         Y & {} & Y'\;
         }
\hspace{3ex}
\xymatrix@C=2ex@R=1ex
         {
          {}
          \\
          R = p' \circ p^{op}, \;\; S = q' \circ q^{op}\,,
         }
$$
and a third $\ell$-relation $R \times S \mr{\theta} G$. 

These data give rise to the following diagrams:
\begin{equation} \label{triangleequation}
\xymatrix@C=1.8ex@R=3ex
        {
         \hspace{-1ex} \Diamond_1 = \Diamond_1(f,g)  
         &&&& \Diamond_2 = \Diamond_2(f,g) 
         &&&& \Diamond = \Diamond(R,S)
        }
\end{equation}
\hspace{-3ex}
\mbox{
$
%\vcenter
       {
        \xymatrix@C=1.4ex@R=3ex
                 {
                  & X \times Y  \ar[rd]^{\lambda} 
                  \\
			        X \times Y' \ar[rd]_{f \times Y'} 
			                    \ar[ru]^{X \times g^{op}} & \equiv & G\,,
			      \\
			       & X' \times Y' \ar[ru]_{\lambda '} 
			      } 
	    }
$
$
%\vcenter
       {
        \xymatrix@C=1.4ex@R=3ex
                 {
                  & X \times Y  \ar[rd]^{\lambda} 
                  \\
			        X' \times Y \ar[rd]_{X' \times g} 
			                    \ar[ru]^{f^{op} \times Y} & \equiv & G\,,
			      \\
			       & X' \times Y' \ar[ru]_{\lambda '} 
			      } 
	    }
$
$
%\vcenter
       {
        \xymatrix@C=1.4ex@R=3ex
                 {
                  & X \times Y  \ar[rd]^{\lambda} 
                  \\
			        X \times Y' \ar[rd]_{R \times Y'} 
			                    \ar[ru]^{X \times S^{op}} & \equiv & G\,,
			      \\
			       & X' \times Y' \ar[ru]_{\lambda '} 
			      } 
	    }
$
}

\vspace{1ex}

expressing the equations: 

\vspace{1ex}

\begin{center} 
$\hspace{2ex} \Diamond_1:\,  
\lambda'\langle f(a),b' \rangle  \,=\, \displaystyle \bigvee_{g(y)=b'} \lambda\langle a,y \rangle \:,$   
$\hspace{2ex} \Diamond_2:\, 
\lambda'\langle a',g(b) \rangle  \,=\, \displaystyle \bigvee_{f(x)=a'} \lambda\langle x,b \rangle$,
\end{center}

and $\hspace{1ex} \Diamond$:
$\hspace{1ex}  
           \displaystyle \bigvee_{(y,\, b')\in S} \lambda\langle a,y \rangle \;\;      
         = \displaystyle \bigvee_{(a,\, x')\in R} \lambda'\langle x',b' \rangle$.
\end{sinnadastandard}

It is clear that diagrams $\Diamond_1$ and  $\Diamond_2$ are particular cases of \mbox{diagram $\Diamond$.} Take $R = f,\; S = g$, then $\Diamond_1(f,g) = \Diamond(f, g)$, and  $R = f^{op},\; S = g^{op}$, then  \mbox{$\Diamond_2(f,g) = \Diamond(f^{op}, g^{op})$.} The general $\Diamond$ diagram follows from these two particular cases. 

\begin{proposition} \label{pre2rhdimpliesdiamante}
Let $R$, $S$ be any two spans connected  by a $\ell$-relation $\theta$ as above. If  $\Diamond_1(p',q')$ and $\Diamond_2(p,q)$ hold, then so does 
$\Diamond(R,S)$.
\end{proposition}
\begin{proof}
We use the elevators calculus, see appendix \ref{ascensores}:
$$
\xymatrix@C=-1.5ex
         {
          \\
             X \did & & Y' \dcellb{\!\!S^{op}}
          \\              
             X & & Y
          \\
            & G \cl{\lambda} 
         }
\xymatrix@R=10ex{ \\ \;\;=\;\; \\}
\xymatrix@C=-1.5ex
         {
             X \did & & Y' \dcellb{\!\!q'^{op}}
          \\
             X \did & & S  \dcell{q}
          \\
             X & & Y
          \\
           & G \cl{\lambda}
         }
\xymatrix@R=10ex{ \\ \;\;\stackrel{\Diamond_2}{=}\;\; \\}
\xymatrix@C=-1.5ex
         {
             X \did & & Y' \dcellb{\!\!q'^{op}}
          \\
             X \dcellb{\!\!p^{op}} & & S  \did
          \\
             R & & S
          \\
            & G \cl{\theta}
         }
\xymatrix@R=10ex{ \\ \;\;=\;\; \\}
\xymatrix@C=-1.5ex
         {
             X \dcellb{\!p^{op}} & & Y' \did
          \\
             R \did  & & Y' \dcellb{\!\!q'^{op}}
          \\
	      R & & S
          \\
            & G \cl{\theta}
         }
\xymatrix@R=10ex{ \\ \;\;\stackrel{\Diamond_1}{=}\;\; \\}
\xymatrix@C=-1.5ex
         {
             X \dcellb{\!p^{op}} & & Y' \did
          \\
             R \dcell{\!p'} & & Y' \did
          \\
             X' & & Y'
          \\
           & G \cl{\lambda'}
         }
\xymatrix@R=10ex{ \\ \;\;=\;\; \\}
\xymatrix@C=-1.5ex
         {
          \\
             X \dcell{R} & & Y' \did
          \\              
             X' & & Y'
          \\
           & G \cl{\lambda'}
         }
$$ 
\end{proof}
Two maps $X \mr{f} X'$, $Y \mr{g} Y'$ also give rise to the following diagram:
$$
\rhd = \rhd(f,g): \hspace{5ex}
\vcenter
        {
         \xymatrix@C=3ex@R=1ex
                   {
                    X \times Y  \ar[rrd]^{\lambda} 
                                \ar[dd]_{f \times g}
                    \\
                    {} & \hspace{-4ex} {}^{\!\!\geq} & G \,.
                    \\
                    X' \times Y'  \ar[rru]_{\lambda '}
                   }
        }
$$
\begin{proposition} \label{diamondimpliesrhd}
If either $\Diamond_1(f,\,g)$ or $\Diamond_2(f,\,g)$ holds, then so does $\;\rhd(f,\,g)$.
\end{proposition}
\begin{proof}
$\lambda\langle a,b \rangle  \leq \displaystyle \bigvee_{g(y)=g(b)} \lambda\langle a,y \rangle  =
  \lambda'\langle f(a),g(b) \rangle $ using $\Diamond_1$. Clearly a symmetric arguing holds
  using $\Diamond_2$. 
\end{proof}
\noindent {\bf In the rest of this section $G$ is assumed to be a locale.}

Consider the following axioms:
\begin{sinnadaitalica} {\bf Axioms on a $\ell$-relation} \label{bijection}

 \vspace{1ex}

$ed) \; \bigvee_{y\in Y} \;\lambda\langle a,y \rangle  = 1$,  
\hspace{6ex} for each a \hfill (everywhere defined).

\vspace{1ex}

$uv) \; \lambda\langle x,b_1 \rangle  \wedge \lambda\langle x,b_2 \rangle  = 0$, \hspace{2ex} for each $x, b_1 \neq b_2$  \hfill (univalued).

\vspace{1ex}

$su) \; \bigvee_{x \in X} \;\lambda\langle x,b \rangle  = 1$, 
\hspace{6ex} for each b \hfill (surjective). 

\vspace{1ex}

$in) \; \lambda\langle a_1,y \rangle  \wedge \lambda\langle a_2,y \rangle  = 0$, \hspace{2ex} for each $y, a_1 \neq a_2$  \hfill (injective).
\end{sinnadaitalica}

Clearly any morphism of locales $G \to H$ preserves these four axioms.

\vspace{1ex}

A $\ell$-relation $\lambda$ is a \emph{$\ell$-function} if and only if satisfies
axioms $ed)$ and $uv)$. We say that a $\ell$-relation is a \emph{$\ell$-opfunction}
when it satisfies   
axioms $su)$ and $in)$. Then a $\ell$-relation is a \emph{$\ell$-bijection} if and only
if it is a $\ell$-function and a $\ell$-opfunction.  

\begin{sinnadastandard} \label{productrelationdef}
Given two $\ell$-relations, $X \times Y \mr{\lambda} G$, $X' \times
Y'\mr{\lambda'}G$, the product $\ell$-relation $\lambda \boxtimes \lambda'$ is defined by the composition

\begin{center}
$X \times X' \times Y \times Y' 
 \mr{X \times \psi \times Y'} X \times Y \times X' \times Y' 
 \mr{\lambda \times \lambda'} G \times G \mr{\wedge} G$ 
\end{center}

\begin{center}
$
(\lambda \boxtimes \lambda')  \langle (a,a'),(b,b') \rangle  = \lambda  \langle a,b \rangle  \wedge \lambda'  \langle a',b' \rangle$. 
\end{center}
\end{sinnadastandard}
The following is immediate and straightforward:
\begin{proposition} \label{productrelation}
Each axiom in \ref{bijection} for $\lambda$ and $\lambda'$ implies the respective axiom for the product $\lambda \boxtimes \lambda'$. \cqd 
\end{proposition}

\begin{proposition} \label{infessup}
We refer to \ref{diagramadiamante12}: If equations $\Diamond_1(p,q)$ and $\Diamond_1(p',q')$ hold, and $\theta$ satisfies $uv)$, then equation 1) below holds. \mbox{Symmetrically,} if  
$\Diamond_2(p,q)$ and $\Diamond_2(p',q')$ hold, and $\theta$ satisfies 
$in)$, then equation 2) below holds.
\begin{center}
$
1) \hspace{5ex}
    \lambda \langle p(r), b \rangle \wedge 
    \lambda' \langle p'(r), b' \rangle  \,=\, 
    \displaystyle\bigvee_{\substack{q(v) = b \\ q'(v) = b'}} \theta 
    \langle  r, v \rangle.
$
\vspace{1ex}
$
2) \hspace{5ex}
    \lambda \langle a, q(s) \rangle \wedge 
    \lambda' \langle  a', q'(s) \rangle  \,=\, 
    \displaystyle\bigvee_{\substack{p(u) = a \\ p'(u) = a'}} \theta 
    \langle  u, s \rangle.
$
\end{center}
\end{proposition}
\begin{proof}
We only prove the first statement, since the second one clearly has a symmetric proof.
\vspace{1ex}

$
\lambda \langle p(r), b \rangle \wedge 
    \lambda' \langle p'(r), b' \rangle   
    \; \stackrel{\Diamond_1}{=} \; \displaystyle \bigvee_{q(v)=b} \theta \langle r,v \rangle 
\,\wedge  
\displaystyle \bigvee_{q'(w)=b'} \theta \langle r,w \rangle  \; = \; 
$

\hfill
$
= \; \displaystyle\bigvee_{\substack {q(v) = b \\ q'(w) = b'}} \theta 
    \langle  r, v \rangle  \wedge \theta \langle  r, w \rangle \; 
    \stackrel{uv)}{=} \;
 \displaystyle\bigvee_{\substack {q(v) = b \\ q'(v) = b'}} \theta 
    \langle  r, v \rangle.
$
\end{proof}   

We study now the validity of the reverse implication in proposition \ref{diamondimpliesrhd}.

\begin{proposition} \label{rhdimpliesdiamond} 
We refer to \ref{diagramadiamante12}:
${}$

1) If $\lambda$ is ed) and $\lambda'$ is uv) (in particular, if they are $\ell$-functions), then $\rhd(f,g)$ implies $\Diamond_1(f,g)$. 
 
\vspace{1ex}

2) If $\lambda$ is su) and $\lambda'$ is in) (in particular, if they are $\ell$-opfunctions), then $\rhd(f,g)$ implies $\Diamond_2(f,g)$. 
\end{proposition}
\begin{proof}
We prove $1)$, a symmetric proof yields $2)$.

$\lambda'\langle f(a),b' \rangle  \stackrel{ed)_{\lambda}}{\;=\;} \lambda'\langle f(a),b' \rangle  \wedge \bigvee_{y} \lambda\langle a,y \rangle  \;=\; \bigvee_{y}
\lambda'\langle f(a),b' \rangle  \wedge \lambda\langle a,y \rangle  \stackrel{(*)}{\;=\;}$ 

\vspace{1ex}

{\hfill $\bigvee_{g(y)=b'} \lambda'\langle f(a),b' \rangle  \wedge \lambda\langle a,y \rangle  \stackrel{\rhd}{\;=\;} \bigvee_{g(y)=b'} \lambda\langle a,y \rangle ,$}

\vspace{1ex}

\noindent where for the equality marked with $(*)$ we used that if $g(y) \neq b'$ then \mbox{$\lambda'\langle f(a),b' \rangle  \wedge \lambda\langle a,y \rangle \stackrel{\rhd}{\leq} \lambda'\langle f(a),b' \rangle  \wedge \lambda'\langle f(a),g(y) \rangle  \stackrel{uv)_{\lambda'}}{=} 0$.}
\end{proof}

More generally, consider two spans as in \ref{diagramadiamante12}. We have the following \mbox{$\rhd$ diagrams:} 
\begin{equation} \label{2rhd}
         \xymatrix@C=3ex@R=1ex
                   {
                    R \times S  \ar[rrd]^{\theta} 
                                \ar[dd]_{p \times q}
                    \\
                    {} & \hspace{-4ex} {}^{\!\!\geq} & G\,,
                    \\
                    X \times Y  \ar[rru]_{\lambda}
                   } 
\hspace{5ex}       
         \xymatrix@C=3ex@R=1ex
                   {
                    R \times S  \ar[rrd]^{\theta} 
                                \ar[dd]_{p' \times q'}
                    \\
                    {} & \hspace{-4ex} {}^{\!\!\geq} & G \,.
                    \\
                    X' \times Y'  \ar[rru]_{\lambda '}
                   }
\end{equation}

\begin{proposition} \label{2rhdimpliesdiamante}

We refer to \ref{diagramadiamante12}: Assume that  
$\lambda$ is in), 
$\lambda '$ is uv), and that the 
$\rhd(p,\,q)$, $\rhd(p',\,q')$ diagrams hold. Then  if 
$\theta$ is ed) and su), diagram $\Diamond(R,\,S)$ holds.
\end{proposition}
\begin{proof}
Use proposition \ref{rhdimpliesdiamond} twice: First with $f =  p'$, $g = q'$, $\lambda = \theta$, 
$\lambda' = \lambda'$ to have $\Diamond_1(p',\,q')$. Second with $f = p$, $g = q$, $\lambda = \theta$, 
$\lambda' = \lambda$ to have 
$\Diamond_2(p,\,q)$. Then use proposition \ref{pre2rhdimpliesdiamante}.
\end{proof}

\begin{remark} \label{productistheta}
Note that the diagrams $\rhd$ in \ref{2rhd} mean that \mbox{$\theta \leq \lambda \boxtimes \lambda' \circ (p, p') \times (q, q')$} (see \ref{productrelationdef}). In particular, when $G$ is a locale, there is always a $\ell$-relation $\theta$ in \ref{diagramadiamante12}, which may be taken to be the composition 
\mbox{$R \times S \mr{(p, p') \times (q, q')} X \times X' \times Y \times Y' 
\mr{\lambda \boxtimes \lambda'} G$}. However, it is important to consider an arbitrary $\ell$-relation $\theta$ (see propositions \ref{trianguloesdiamante} and \ref{supisloc}).
\end{remark}
\begin{proposition} \label{diamanteimplies2rhd}
We refer to \ref{diagramadiamante12}: Assume that $R$ and $S$ are relations, that $\lambda$, $\lambda '$ are $\ell$-bijections, and that  
$\rhd(p,\,q)$, $\rhd(p',\,q' )$ in \eqref{2rhd} hold. Take 
\mbox{$\theta=\lambda \boxtimes \lambda' \circ (p, p') \times (q, q')$}. Then,
if $\Diamond(R,\,S)$ holds, $\theta$ is a $\ell$-bijection.
\end{proposition}
\begin{proof}
We can safely assume  $R \subset X \times X'$ and 
$S \subset Y \times Y'$, and \mbox{$\lambda \boxtimes \lambda' \circ (p, p') \times (q, q')$} to be the restriction of $\lambda \boxtimes \lambda'$ to $R \times S$.
From the $\rhd$ diagrams \eqref{2rhd}  we easily see that axioms $uv)$ and $in)$ for $\theta$ follow from the corresponding axioms for 
$\lambda$ and $\lambda'$. We prove now axiom $ed)$, axiom $su)$ follows in a symmetrical way. 
Let $(a,a') \in R$, we compute:

\vspace{1ex}

$
\displaystyle\bigvee_{(y,y') \in S} \theta \langle (a,a'), (y,y') \rangle 
\,=\, 
\bigvee_{y'} \bigvee_{(y,y') \in S} \lambda \langle a,y \rangle \wedge \lambda' \langle a',y'\rangle
\stackrel{\Diamond}{=}$ 

\hfill 
$
\stackrel{\Diamond}{=} 
\displaystyle\bigvee_{y'} \bigvee_{(a,x') \in R} \lambda' \langle x',y' \rangle \wedge \lambda' \langle a',y'\rangle 
\geq 
\bigvee_{y'} \lambda' \langle a',y'\rangle 
\stackrel{ed)}{=} 1
$
\end{proof}

We found convenient to combine \ref{2rhdimpliesdiamante} and \ref{diamanteimplies2rhd} into:
\begin{proposition} \label{combinacion}
Let $R \subset X \times X'$, $S \subset Y \times Y'$ be any two relations, and $X \times Y \mr{\lambda} G$, 
$X' \times Y' \mr{\lambda'} G$ be $\ell$-bijections.  Let $R \times S \mr{\theta} G$ be the restriction of $\lambda \boxtimes \lambda'$ to $R \times S$. Then, 
$\Diamond(R, S)$ holds if and only if $\theta$ is a $\ell$-bijection. \cqd
\end{proposition}

\section{On $\rhd$ and $\Diamond$ cones}

We consider a pointed topos $\cc{E}ns \mr{f} \cc{E}$, with inverse image $f^* = F$. 
\begin{sinnadastandard} \label{rel}
 Let $\cc{R}el(\cc{E})$ be the category of relations in $\cc{E}$. $\cc{R}el(\cc{E})$ is a symmetric monoidal category with tensor product given by the cartesian product in 
$\cc{E}$ (which is not cartesian in $\cc{R}el(\cc{E})$). Every object $X$ has a dual, and it is self dual, the unit and the counit of the duality are both given by the diagonal relation $\Delta \subset X \times X$ (see \ref{tensoriso}). There is a faithful functor \mbox{$\cc{E} \rightarrow \cc{R}el(\cc{E})$}, the identity on objects and the graph on arrows, we will often abuse notation and identify an arrow with its graph. The functor $\cc{E} \mr{F} \cc{E}ns$ has an extension 
\mbox{$\cc{R}el(\cc{E}) \mr{\cc{R}el(F)} \cc{R}el$,} if $R \subset X \times Y$ is a relation, then $FR \subset FX \times FY$, and 
$\cc{R}el(F)$ is in this way a tensor functor.
We have a commutative diagram:
$$
\xymatrix
        {
           \cc{E} \ar[r] \ar[d]_F
         & \cc{R}el(\cc{E}) \ar[d]^{T}
         \\
           \cc{E}ns \ar[r] 
         & \cc{R}el \; \ar@{^(->}[r]^{\ell} 
         & Sup 
         & {(\text{where} \;  T = \cc{R}el(F))} 
        }
$$
\end{sinnadastandard} 
\begin{sinnadaitalica} \label{FequivalenceiffRel(F)so}
It can be seen that          
$F$ is an equivalence if and only if $T$ is so.  
\cqd
\end{sinnadaitalica} 
Note that on objects $TX = FX$ and on arrows in $\cc{E}$, $T(f) = F(f)$. Since $F$ is the inverse image of a point, the diagram of $F$ is a cofiltered category, $T(X \times Y) = TX \times TY$, if $C_i \to X$ is an epimorphic family in 
$\cc{E}$, then $TC_i \to TX$ is a surjective family of sets. If $R$ is an arrow in $\cc{R}el(\cc{E})$, \mbox{$T(R^{op}) = (TR)^{op}$.} 

Let $H$ be a sup-lattice furnished with a $\ell$-relation $TX \times TX \mr{\lambda_{X}} H$ for each $X \in \cc{E}$. Each arrow $X \mr{f} Y$ in 
$\cc{E}$ and each arrow 
$X \mr{R} Y$ in $\cc{R}el(\cc{E})$ (i.e relation 
$R \mmr{} X \times Y$, $R \mr{\pi_1} X$, $R \mr{\pi_2} Y$ in $\cc{E}$), determine the following diagrams:
%\begin{equation} \label{AutF}
$$
\xymatrix@C=4ex@R=3ex
        {
         FX \times FX \ar[rd]^{\lambda_{X}}  
                      \ar[dd]_{F(f) \times F(f)}  
        \\
         {} \ar@{}[r]^(.3){\geq}
         &  H\,,    
        \\
         FY \times FY \ar[ru]^{\lambda_{Y}} 
         } 
\hspace{3ex}
\xymatrix@C=4ex@R=3ex
        {
         & TX \times TX \ar[rd]^{\lambda_{X}}  
        \\
           TX \times TY 
               \ar[rd]_{TR \times TY \hspace{2.5ex}} 
			   \ar[ru]^{TX \times TR^{op}\hspace{2.5ex}} 
	     & \equiv 
         & H\,. 
        \\
         & TY \times TY \ar[ru]^{\lambda_{Y}} 
         }
$$ 
%\end{equation}
We say that $TX \times TX \mr{\lambda_{X}} H$ is a \emph{$\rhd$-cone} if the $\, \rhd(F(f),\, F(f)) \,$ diagrams  hold, and that it is a 
\emph{$\Diamond$-cone} if the $\,\Diamond(TR,TR)\,$ diagrams hold. Similarly we talk of \emph{$\Diamond_1$-cones} and \emph{$\Diamond_2$-cones} if the 
$\Diamond_1(F(f),\, F(f))$ and $\Diamond_2(F(f),\, F(f))$ \mbox{diagrams} hold. We will abbreviate $\Diamond(R) = \Diamond(TR,TR)$, and similarly $\rhd(f)$, $\Diamond_1(f)$ and $\Diamond_2(f)$. 
If $H$ is a locale and the 
$\lambda_X$ are $\ell$-bijections, we say that we have a $\Diamond$-cone or a $\rhd$-cone \textit{of $\ell$-bijections}.

\begin{proposition} \label{dim1ydim2esdim}
A family $TX \times TX \mr{\lambda_{X}} H$ of $\ell$-relations is a $\Diamond$-cone if and only if it is both a $\Diamond_1$ and a $\Diamond_2$-cone.
\end{proposition}
\begin{proof}
Use proposition \ref{pre2rhdimpliesdiamante} with $R = TR$, $S = TR$, $p = p' = \pi_1$, $q = q' = \pi_2$, $\lambda = \lambda_X$,  
$\lambda' = \lambda_Y$, and $\theta = \lambda_R$. Then, $\Diamond_1(\pi_2)$ and $\Diamond_2(\pi_1)$ imply $\Diamond(R)$
\end{proof}
\begin{proposition} \label{trianguloesdiamante}
Any $\rhd$-cone $TX \times TX \mr{\lambda_{X}} H$ of $\ell$-bijections with values in a locale $H$ is a $\Diamond$-cone (of $\ell$-bijections).
\end{proposition}
\begin{proof}
Given any relation $R \mmr{} X \times Y$, consider proposition \ref{2rhdimpliesdiamante} with \mbox{$\lambda = \lambda_X$,} $\lambda' = \lambda_Y$, and $\theta = \lambda_R$.
\end{proof}

\begin{definition} \label{comp}
Let $TX \times TX \mr{\lambda_{X}} H$ be a $\Diamond$-cone with values in a commutative algebra $H$ in $Sup$, with multiplication $*$ and unit $u$. 
We say that it is \emph{compatible} if the following equations hold:
$$
\lambda_X \langle a,\,a' \rangle *  \lambda_Y \langle b,\,b' \rangle = 
\lambda_{X \times Y}\langle (a,\,b),\,(a',\,b') \rangle \,,
\hspace{3ex} \lambda_1(*,\,*) = u. 
$$
\end{definition}

Any compatible $\Diamond$-cone wich covers $H$ forces $H$ to be a locale, and such a cone is necessarily a cone of $\ell$-bijections (and vice versa). We examine this now:

Given a compatible cone, consider the diagonal $X \mr{\Delta} X \times X$, the arrow $X \mr{\pi} 1$, and the following $\Diamond_1$ diagrams:
%\begin{equation}\label{dosdiag}
$$
\xymatrix@R=4ex@C=2ex
        { 
         & TX \!\times\! TX \ar[rd]^{\lambda_X} 
         & & & TX \!\times\! TX \ar[rd]^{\lambda_X} 
        \\
	     TX \!\times\! TX \!\times\! TX 
	                      \ar[ru]^{TX \!\times\! \Delta^{op}} 
	                      \ar[rd]_{\Delta \!\times\! TX \!\times\! TX \ \ } 
	     & \equiv 
	     & \hspace{0ex} H, 
	     & TX \!\times\! 1 \ar[ru]^{TX \!\times\! \pi^{op}} 
	                  \ar[rd]_{\pi \!\times\! 1} 
	     & \!\! \equiv 
	     & \hspace{0ex} H.
	    \\
	     & TX \!\times\! TX \!\times\! TX \!\times\! TX 
	                          \ar[ru]_{\lambda_{X \!\times\! X}} 
	     & & & 1 \!\times\! 1 \ar[ru]_{\lambda_1} 
	    }
$$
%\end{equation}
Let $a,\, b_1,\, b_2 \in TX$, and let $b$ stand for either $b_1$ or $b_2$. Chasing $(a,\,b_1,\,b_2)$ in the first diagram and $(a,*)$ in the second it follows: 

\vspace{1ex}

$
(1) \;\;
\lambda_X\langle a,\,b_1\rangle * \lambda_X\langle a,\,b_2\rangle \,=
\lambda_{X \times X}\langle (a,\,a),\,(b_1,\,b_2) \rangle \, = \;
\delta_{b_1 = b_2} \, \lambda_X\langle a,\,b \rangle. 
$

\vspace{1ex} 

$
 (2) \;\; 
 \lambda_X(a,\,b) \;\leq\; 
 \bigvee_x \lambda_X\langle a,\,x\rangle \;=\; 
 \lambda_1(*,*) = u.
$
\begin{proposition} \label{compislocale}
Let $H$ be a commutative algebra, and $TX \times TX \mr{\lambda_{X}} H$ be a compatible $\Diamond$-cone such that the elements of the form $\lambda_X(a,\,a')$, $a,\, a' \in TX$ are algebra generators. Then   
$H$ is a locale. 
\end{proposition}
\begin{proof}
We have to prove that for all $w \in H$, $w * w = w$ and $w \leq u$, see \ref{charlocales}. It is enough to prove it for $w = \lambda_X(a,\,b)$, which are precisely equations (1) and (2) above.
\end{proof}

\begin{proposition}\label{Diamondisbijection}  
A $\Diamond$-cone $TX \times TX \mr{\lambda_{X}} H$ with values in a locale $H$ is compatible if and only if it is a $\Diamond$-cone of $\ell$-bijections.
\end{proposition}
\begin{proof}
($\Rightarrow$): Clearly equations (1) and (2) above are the axioms uv) and ed) for $\lambda_X$. Axioms in) and su) follow by the symmetric argument using the corresponding $\Diamond_2$ diagrams.

($\Leftarrow$) $u = 1$ in $H$, so the second equation in definition \ref{comp} is just axiom $ed)$ (or $su)$) for $\lambda_1$. To prove the first equation we do as follows: 

Consider the projections $X \times Y \mr{\pi_1} X$, $X \times Y \mr{\pi_2} Y$. The $\Diamond_1(\pi_1)$  and $\Diamond_1(\pi_2)$ diagrams express the equations:

\vspace{1ex}

$
 \lambda_X\langle a,a'\rangle = \bigvee_y \lambda_{X \times Y} \langle (a,b),(a',y) \rangle, \;\;  \lambda_Y \langle b,b'\rangle = \bigvee_x \lambda_{X \times Y} \langle (a,b),(x,b') \rangle.
$

\vspace{1ex}

Taking the infimum of these two equations:

\vspace{1ex}

\noindent
$
\lambda_X \langle a,a' \rangle \wedge \lambda_Y \langle b,b' \rangle  \;=\;\bigvee_{x,y} \lambda_{X \times Y} \langle (a,b),(a',y) \rangle  \wedge \lambda_{X \times Y} \langle (a,b),(x,b') \rangle \stackrel{(*)}{\;=\;} 
$

\vspace{1ex}

\noindent
$
\stackrel{(*)}{\;=\;} \lambda_{X \times Y}\langle(a,b),(a',b')\rangle$, as desired 
($\stackrel{(*)}{\;=\;}$ justified by $uv)$ for $\lambda_{X \times Y}$). 
\end{proof}

\begin{proposition} \label{supisloc}
Let $TX \times TX \mr{\lambda_{X}} H$ be a $\Diamond$-cone of $\ell$-bijections such that the elements of the form $\lambda_X(a,\,a')$, $a,\, a' \in TX$ are locale generators. Then, any linear map $H \mr{\sigma} G$ into another 
$\Diamond$-cone of $\ell$-bijections, 
\mbox{$TX \times TX \mr{\lambda_{X}} G$,} satisfying $ \sigma \lambda_X = \lambda_X$, preserves infimum and $1$, thus it is a locale morphism.
\end{proposition}
\begin{proof}
By axiom $ed)$ for $\lambda_1$, in both locales 
$\lambda_1(*,\,*) = 1$. Since $\sigma \lambda_1 = \lambda_1$, this shows that $\sigma$ preserves $1$. 

To show that infima are preserved it is enough to prove that infima of the form $\lambda_X \langle a, a' \rangle  \wedge 
\lambda_Y \langle b, b' \rangle$, $a,\,a' \in TX,\; b,\, b' \in TY$ are preserved. 
Take  
$
\vcenter
       {
        \xymatrix@R=0ex@C=3ex
                {
                 & (X,\, a) 
                \\ (Z,\, c) \ar[ru]^{f} \ar[rd]_{g} 
                \\ 
                 & (Y,\, b)
                 } 
       }
$ in the diagram of $F$. %(recall that $T = F$).
Then, by \mbox{proposition} \ref{infessup} with \mbox{$\lambda = \lambda_X$,} $\lambda' = \lambda_Y$, and $\theta = \lambda_Z$, it follows that the equation  
\mbox{$\lambda_X \langle a, a' \rangle  \wedge 
\lambda_Y \langle b, b' \rangle  = 
\displaystyle\bigvee_{T(f)(z) = a' \,,\, T(g)(z) = b'} \lambda_Z \langle  c, z \rangle$} holds in both locales. The proof finishes using that  $\sigma$ preserves suprema and $\sigma\lambda_Z=\lambda_Z$.
\end{proof}

Consider now a (small) site of definition $\cc{C} \subset \cc{E}$ of the topos $\cc{E}$. Suitable cones defined over $\cc{C}$ can be extended to $\cc{E}$. More precisely:
\begin{proposition} $ $ \label{extension}

1) Let $TC \times TC \mr{\lambda_C} H$ be a $\Diamond_1$-cone (resp. 
a $\Diamond_2$-cone) over $\cc{C}$. Then, $H$ can be (uniquely) furnished with $\ell$-relations $\lambda_X$ for all objects $X \in \cc{E}$ in such a way to determine a $\Diamond_1$-cone (resp. a $\Diamond_2$-cone) over $\cc{E}$. 
 
2) If $H$ is a locale and all the $\lambda_C$ are $\ell$-bijections, so are all the $\lambda_X$.
\end{proposition}
\begin{proof} 

1)  
Let $X \in \cc{E}$ and $(a,\,b) \in TX \times TX$.  Take $C \mr{f} X$ and $c \in TC$ such that $a = T(f)(c)$. If $\lambda_X$ were defined so that the 
$\Diamond_1(f)$ diagram commutes, the equation 
$$
(1) \hspace{2ex} \lambda_X\langle a,\,b\rangle  = \bigvee_{T(f)(y)=b} \lambda_C\langle c,\,y\rangle
$$ 
should hold (see \ref{triangleequation}). We define $\lambda_X$ by this equation. This definition is independent of the choice of $c, \, C,$ and $f$. In fact, let $D \mr{g} X$ and $d \in TD$ such that 
$a = T(g)(d)$, and take 
$(e,\, E)$ in the diagram of $F$, $E \mr{h} C$, $E \mr{\ell} D$ such that 
$T(h)(e) = c$, $T(\ell)(e) = d$ and $T(fh)=T(g\ell)$. Then we compute
$$
\bigvee_{T(f)(y)=b} \lambda_C\langle c,\,y\rangle \stackrel{\Diamond_1(h)}{=}
\bigvee_{T(f)(y)=b} \;\;\bigvee_{T(h)(w)=y} \lambda_E\langle e,\,w\rangle \;= 
\bigvee_{T(fh)(w)=b} \lambda_E\langle e,\,w\rangle.
$$
From this and the corresponding computation with $d, \, D,$ and $\ell$ it follows:
$$
\bigvee_{T(f)(y)=b} \lambda_C\langle c,\,y\rangle = \bigvee_{T(g)(y)=b} \lambda_{D}\langle d,\,y\rangle.
$$
Given $X \mr{g} Y$ in $\cc{E}$, we check that the  
$\Diamond_1(g)$ diagram commutes: Let $(a,\,b) \in TX \times TY$, take $C \mr{f} X$ such that $a = T(f)(c)$, and let \mbox{$d = T(g)(a) = T(gf)(c)$.} Then:
$$
\lambda_Y \langle d,b \rangle = 
\hspace{-1ex} \bigvee_{T(gf)(z)=b} \lambda_C \langle c,z \rangle = 
\hspace{-1ex} \bigvee_{T(g)(x)=b} \bigvee_{T(f)(z)=x} \lambda_C \langle c,z \rangle = 
\hspace{-1ex} \bigvee_{T(g)(x)=b} \lambda_X \langle a,b \rangle.
$$
Clearly a symmetric argument can be used if we assume at the start that the $\Diamond_2$ diagram commutes. In this case, $\lambda_X$ would be defined by:                
$$
(2) \hspace{2ex}  \lambda_X\langle a,\,b\rangle  = \bigvee_{T(f)(y)=a} \lambda_C\langle y,\,c\rangle
$$ 
with $T(f)(c) = b$. 

If the $TC \times TC \mr{\lambda_C} H$ form a $\Diamond_1$ and a 
$\Diamond_2$ cone, definitions (1) and (2) coincide. In fact, since they are both independent of the chosen $c$, it follows they are equal to 
$$
\bigvee_{T(f)(y)=b, \;T(f)(c)=a} \lambda_C\langle c,\,y\rangle 
\hspace{2ex} = \hspace{2ex} 
\bigvee_{T(f)(y)=a, \;T(f)(c)=b} \lambda_C\langle y,\,c\rangle 
$$ 

\vspace{1ex}

2) It is straightforward and we leave it to the reader.
\end{proof}

It is worthwhile to consider the case of a locally connected topos. In this case it clearly follows from the above (abusing notation) that given $a, \; b \in TX$, if $a, \, b$ are in the same connected component 
$C \subset X$,  $a, \; b \in TC$, then 
$\lambda_X(a,\,b) = \lambda_C(a,\,b)$, and if they are in different connected components, then $\lambda_X(a,\,b) = 0$. When the topos is atomic and $H = Aut(F)$ (see \ref{Aut(F)}), the reverse implication holds, namely, if $\lambda_X(a,\,b) = 0$, then $a, \; b$ must be in different connected components (Theorem \ref{keyLGT}, 1)).

\section{The isomorphism $Cmd_0(G) = Rel(\beta^G$)} \label{Cmd_0}

The purpose of this section is to establish an isomorphism of categories between $Cmd_0(G)$ and $Rel(\beta^G)$, where $G$ is a 
fixed localic group, or, what amounts to the same thing,
an idempotent Hopf algebra in the monoidal category $Sup$ of sup-lattices, as we explained in section \ref{background}. %We refer to \ref{rel}. 

\begin{sinnadastandard} {\bf The category $Cmd_0(G)$.} 
\end{sinnadastandard}
As for any coalgebra, a \emph{comodule} structure over $G$ in $Sup$ is a 
sup-lattice  $S \in Sup$ together with a map $S \stackrel{\rho}{\rightarrow} G \otimes S$ satisfying the coaction axioms:
\begin{equation} \label{coaction}
(G \otimes \rho) \circ \rho =  (w \otimes S) \circ \rho, \;\; and \;\; 
 (e \otimes S) \circ \rho = \; \cong_S.
\end{equation}
where $w$, $e$ are the comultiplication and the counit of $G$, and $\cong_S$ is the isomorphism $2 \otimes S \cong S$.

A \emph{comodule morphism} between two comodules is a map which makes the
usual diagrams commute (see \cite{JS}). We define the category $Cmd_0(G)$ as the full subcategory 
with objects the comodules of the form $S = \ell X$, for any set $X$. If we forget the comodule structure we have a faithful functor 
$$Cmd_0(G) \mr{T} Sup_0 = \cc{R}el.$$

\begin{sinnadastandard} {\bf The category $\beta^G$.} \end{sinnadastandard}
The construction of the category $\beta^G$ of sets furnished with an action of $G$ (namely,
the \emph{classifying topos} of $G$) requires
some considerations (for details see \cite{D1}).
Given a set $X$, we define the locale $Aut(F)$ to be the universal \mbox{$\ell$-bijection} in the category of locales, $X \times X \mr{\lambda} Aut(F)$. It is constructed in two steps: first consider the free locale on $X \times X$, $X \times X \mr{\jmath} Rel(X)$. Clearly it is the universal $\ell$-relation in the category of locales. Second, $Rel(X) \mr{} Aut(X)$ is determined by the topology generated by the covers that force the four axioms in \ref{bijection} (see \cite{W}, \cite{D1}).
 Notice that it follows by definition that the points of the locales $Rel(X)$ and $Aut(X)$ are the relations and the bijections of the set $X$.
Given $(x,\,y) \in X \times X$, 
we will denote $\langle x\,|\,y \rangle  = \jmath \langle x,y \rangle = \lambda \langle x,y \rangle$ indistinctly in both cases. We abuse notation and omit to indicate the associated sheaf morphism  $Rel(X) \mr{}  Aut(X)$. The elements of the form $\langle x\,|\,y \rangle$ generate both locales by taking arbitrary suprema of finite infima. 

It is straightforward to check that the following maps are $\ell$-bijections.
\begin{multline} \label{algebrastructure}
w: \; X \times X  \mr{}  Aut(X) \otimes Aut(X),\;\;
w \langle x \:|\:y \rangle  \;=\;
 \bigvee\nolimits_{z} \; \langle  x \:|\:z  \rangle \otimes \langle  z \:|\:y  \rangle \,, 
\\
\hspace{-36ex} e: \; X \times X \mr{} 2, \;\;  e \langle  x\,|\,y  \rangle  \;=\;  \delta_{x=y}.
\\
\hspace{-28ex} \iota: \; X \times X \mr{} Aut(X), \;\;
\iota \langle  x \:|\:y  \rangle  \;=\; \langle  y \:|\:x  \rangle \,.
\\
\end{multline}

 It follows (from the universal property) that they determine locale morphisms with domain $Aut(X)$. They define a coalgebra structure on the locale $Aut(X)$, which furthermore results a Hopf algebra (or localic group).

\vspace{1ex}

An \emph{action} of a localic group $G$ in a set $X$ is defined as a localic
group morphism $G \mr{\widehat \mu} Aut (X)$. This
corresponds to a Hopf algebra morphism $Aut(X) \mr{\mu} G$, which 
is completely determined by its value on the generators, that is,
a $\ell$-bijection $X \times X \mr{\mu} G$, that in addition satisfies
\begin{equation} \label{morfismo}
w \mu  \;=\; (\mu  \otimes \mu )w \,,\;\;\;\;\;\;
e \mu  \;=\; e \,,\;\;\;\;\;\; 
\mu \iota  \;=\; \iota \mu. 
\end{equation}
(the structures in both Hopf algebras are indicated with the same letters).

As we shall see in Proposition \ref{monoidaction}, the equation $\mu \iota  \;=\; \iota \mu$ follows from the other two. That is, any action of $G$ viewed as a monoid is automatically a group action.

Given two objects $X, X' \in \beta^G$, a \emph{morphism} between them is a function between the sets $X \mr{f} X'$ 
satisfying $\mu\langle a|b \rangle  \leq \mu'\langle f(a)|f(b) \rangle $. Notice that this is a $\;\rhd\;$ diagram as in section \ref{sec:general}.

\vspace{1ex}

If we forget the action we have a faithful functor $\beta^G \mr{F} \cc{E}ns$  (which is the inverse image of a point of the topos, see \cite{D1} Proposition 8.2). Thus, we have a commutative square (see \ref{rel}):
$$
\xymatrix
        {
           \beta^G \ar[r] \ar[d]_F
         & \cc{R}el(\beta^G) \ar[d]^{\cc{R}el(F)}
         \\
           \cc{E}ns \ar[r] 
         & \cc{R}el.
        }
$$
We have the following theorem, that we will prove in the rest of this section.
\begin{theorem} \label{Comd=Rel}
There is an isomorphism of categories making the triangle commutative:
$$
\xymatrix@C=0ex
        {
          Cmd_0(G) \ar[rr]^{=} \ar[rd]_T 
      & & \cc{R}el(\beta^G) \ar[ld]^(.4){\cc{R}el(F)} 
        \\
         & Sup_0 = \cc{R}el.
        }       
$$ 
The identification between relations $R \subset X \times X'$ and linear maps $\ell X \to \ell X'$ lifts to the upper part of the triangle.

\vspace{-2.5ex}

\hfill $\Box$
\end{theorem} 
 
\vspace{1ex}       

Recall that since the functor $F$ is the inverse image of a point, it follows that monomorphisms of $G$-sets are injective maps.
\begin{proposition} \label{mono1}
 Let $f: X \rightarrow X'$ a morphism of $G$-sets. Then for each 
 \mbox{$a, b \in X$}, $$ \mu'  \langle f(a)|f(b) \rangle  = \bigvee_{f(x)=f(b)} \mu  \langle a|x \rangle .$$
In particular, if $f$ is a monomorphism, we have $\mu'  \langle f(a)|f(b) \rangle  = \mu  \langle a|b \rangle $.
\end{proposition}
\begin{proof}
Since the actions are $\ell$-bijections, in particular $\ell$-functions, by  proposition \ref{rhdimpliesdiamond} the $\rhd$ diagram implies the $\Diamond_1$ diagram. The statement follows by taking $(a,f(b)) \in X \times X'$.
\end{proof}
Proposition \ref{mono1} says that the subobjects $Z \mono X$ of an object $X$ in                               
$\beta^G$ are the subsets $Z \subset X$ such that the restriction of the action \mbox{$Z \times Z \subset X \times X \mr{\mu} G$} is an action on $Z$. We have:
\begin{proposition} \label{mono2}
Let $X$ be a $G$-set and $Z \subset X$ any subset. If the restriction of the action to $Z$ is a $\ell$-bijection, then it is already an action.
\end{proposition}
\begin{proof}
We have to check the equations in \ref{morfismo}. The only one that requires some care is the first. Here it is convenient to distinguish notationally as $w_Z$, $w_X$ and $w$ the comultiplications of $Aut(Z)$, $Aut(X)$ and $G$ respectively. By hypothesis we have (1)
$w\mu \langle a|b \rangle = (\mu \otimes \mu) w_X \langle a|b \rangle = \displaystyle \bigvee_{x \in X} \mu \langle 
a|x \rangle \otimes \mu \langle x|b \rangle$. We claim that when 
$a,\, b \in Z$, this equation still holds by restricting the supremum to the $x \in Z$, which is the equation $w\mu \langle a|b \rangle = (\mu \otimes \mu) w_Z$. 
In fact, from axioms $ed)$ and $su)$ for $\mu$ on $Z$ it follows  
(2) $1 = \displaystyle \bigvee_{y,\, z \,\in Z} \mu \langle a|y \rangle \otimes \mu \langle z|b \rangle$. Then, the claim follows by taking the infimum in both sides of equations (1) and (2), and then using the axioms $uv)$ and $in)$ for $\mu$ on $X$.
\end{proof}

\begin{proposition}  \label{monoidaction}
 Given a localic group $G$ and a localic monoid morphism $G \stackrel{\widehat{\mu}}{\rightarrow} Rel(X)$, there exists a unique action of $G$ in $X$ such that
$$ 
\xymatrix@C=4ex@R=4ex
         {
          Rel(X) 
          & & G, \ar[ll]_{\widehat{\mu}} \ar@{-->}[dl]^{\widehat\mu}	
          & \hbox{i.e.} 
          & Rel(X) \ar[rr]^{\mu} \ar[rd] 
          & & G. 
          \\ 	
		  & Aut(X) \ar[ul] 
		  & & & & Aut(X) \ar@{-->}[ru]_{\mu} 
         } 
$$
\end{proposition}
\begin{proof}
 $\mu$ is determined by a $\ell$-relation $X \times X \mr{\mu} G$ preserving $w$ and $e$ (see equations \ref{morfismo}). It factorizes through $Aut(X)$ provided it is a $\ell$-bijection, and the factorization defines an action if it also preserves $\iota$. 

Consider the following commutative diagram
$$
\xymatrix
         {
          & X \times X \ar[r]^(.35){w} \ar[d]^\mu \ar[dl]_e
          & Rel(X) \otimes Rel(X) \ar[d]^{\mu \otimes \mu}
          \\
          2 \ar[dr]_u
          & G \ar[r]^w \ar[l]_e
          & G \otimes G \ar@<-1.5ex>[d]_{\iota \otimes G}
                        \ar@<1.5ex>[d]^{G \otimes \iota}
          \\
          & G
          & G \otimes G. \ar[l]_\wedge
         } 
$$
Chasing an element $(b,\, b) \in X \times X$  all the way down to $G$ using the arrow $G \otimes \iota$ it follows
$\displaystyle \bigvee_{y} \mu  \langle b|y \rangle  \wedge \iota \mu  \langle y|b \rangle  = 1$. Thus, in particular, we have  
\mbox{(1) $\displaystyle\bigvee_{y} \mu  \langle b|y \rangle  = 1$.}
Chasing in the same way an element $(a,\, b)$ with $a \neq b$, but this time using the arrow $\iota \otimes G$, it follows 
\mbox{$\displaystyle \bigvee_{x} \iota  \mu  \langle a|x \rangle  \wedge \mu  \langle x|b \rangle  = 0$. Thus}
\mbox{(2) $\iota  \mu  \langle a|x \rangle  \wedge \mu  \langle x|b \rangle  = 0$ for all $x$.}

\vspace{1ex}

We will see now that $\iota \mu \leq \mu \iota$ (since $\iota^2 = id$, it follows that also $\mu \iota \leq \iota \mu$).

\vspace{1ex}
 
$ \iota  \mu  \langle a|b \rangle  \stackrel{(1)}{=} \iota  \mu  \langle a|b \rangle  \wedge \displaystyle\bigvee_{y} \mu  \langle b|y \rangle  = $
$ \displaystyle\bigvee_{y} \iota  \mu  \langle a|b \rangle  \wedge \mu  \langle b|y \rangle  \stackrel{(2)}{=} \iota  \mu  \langle a|b \rangle  \wedge \mu  \langle b|a \rangle$, 
since all the other terms in the supremum are $0$. 
Then $\iota  \mu  \langle a|b \rangle  \leq \mu  \langle b|a \rangle = \mu \iota \langle a|b \rangle $. 

\vspace{1ex}

Thus we have $\iota  \mu  \langle a|b \rangle  = \mu \iota \langle a|b \rangle \; ( = \mu  \langle b|a \rangle)$. With this, it is clear from the equations (1) and (2) above that the four axioms \ref{bijection} of a $\ell$-bijection hold. 
\end{proof}

\begin{proposition} 
There is a bijection between the objects of the categories $Cmd_0(G)$ and $Rel(\beta^G)$.
\end{proposition}
\begin{proof}
 Since $(\ell X)^\wedge = \ell X$, we have a bijection of linear maps
$$ 
\xymatrix@R=0.1pc{& \ell X \ar[r]^{\rho} & G \otimes \ell X \\
	    \ar@{-}[rrr] & & & \\
	     & \ell X \otimes \ell X \ar[r]^{\mu } & G.}
$$
As with every duality $(\varepsilon$, $\eta$), $\mu $ is defined as the composition 
$$
 \xymatrix { \mu : \ell X \otimes \ell X \ar[r]^{\rho \otimes \ell X} & G \otimes \ell X \otimes \ell X \ar[r]^(.65){G \otimes \varepsilon} & G.}
$$
And conversely, we construct $\rho$ as the composition
$$
 \xymatrix { \rho : \ell X \ar[r]^(.35){\ell X \otimes \eta} & \ell X \otimes \ell X \otimes \ell X \ar[r]^(.6){\mu  \otimes \ell X} & G \otimes \ell X.}
$$
It is easy to check (for example, using the elevators calculus) that 
that $\rho$ satisfies equations \ref{coaction} if and only if $\mu$ satisfies the first two equations \ref{morfismo} (by proposition \ref{monoidaction}, such a $\mu$ satisfies also the third equation).  
\end{proof}

The product of two $G$-sets $X$ and $X'$ is equipped with the action given by the product $\ell$-relation $\mu \boxtimes \mu'$ (\ref{productrelationdef}),  which is an action by proposition \ref{productrelation}.
 
\vspace{1ex}

An arrow of the category $Rel(\beta^G)$ is a monomorphism 
$R \hookrightarrow X \times X'$, in particular, a relation of sets
$R \subset X \times X'$. 
It follows from propositions \ref{mono1} and \ref{mono2}, that a relation $R \hookrightarrow X \times X'$ in the category $\beta^G$ is the same thing that a relation of sets $R \subset X \times X'$ such that the restriction of the product action to $R$ is still a $\ell$-bijection (on $R$). The following proposition finishes the proof of theorem  \ref{Comd=Rel}.
\begin{proposition} \label{conclaim}
Let $X$, $X'$ be any two $G$-sets, and $R \subset X \times X'$ a relation on the underlying sets. Then, $R$ underlines a monomorphism of \mbox{$G$-sets} $R \hookrightarrow X \times X'$ if and only if the corresponding linear map 
$R: \ell X \rightarrow \ell X'$ is a comodule morphism.
\end{proposition}
\begin{proof}
Let $\theta$ be the restriction of the product action $\mu \times \mu'$ to $R$. We claim that the diagram expressing that $R: \ell X \rightarrow \ell X'$ is a comodule morphism is equivalent to the diagram $\Diamond(R,R)$ in \ref{diagramadiamante12}. The proof follows then by 
\mbox{proposition \ref{combinacion}.}

 \vspace{1ex}
 
{\it proof of the claim}: 
It can be done by chasing elements in the diagrams, or more generally by using the elevators calculus explained in appendix \ref{ascensores}:

The comodule morphism diagram is the equality
\begin{equation} \label{commorph}
\xymatrix@C=-0.3pc@R=1.5pc
          {  
                   \ell X  \ar@2{-}[d] 
           & & &           \ar@{-}[dl] 
                           \ar@{}[d]|{\eta} 
                           \ar@{-}[dr] 
           & & &   \ell X  \ar@<4pt>@{-}'+<0pt,-6pt>[d]  
                           \ar@<-4pt>@{-}'+<0pt,-6pt>[d]^{R}	
           & & &           \ar@{-}[dl] 
                           \ar@{}[d]|{\eta} 
                           \ar@{-}[dr] 
           \\
				  \ell X   \ar@{-}[dr] 
		   &               \ar@{}[d]|{\mu_X } 
		   &      \ell X   \ar@{-}[dl] 
		   & &    \ell X   \ar@<4pt>@{-}'+<0pt,-6pt>[d] 
		                   \ar@<-4pt>@{-}'+<0pt,-6pt>[d]^{R}	
		   &      \hspace{3ex} = \hspace{3ex}
		   &	  \ell X'  \ar@{-}[dr] 
		   &               \ar@{}[d]|{\mu_{X'} } 
		   &      \ell X'  \ar@{-}[dl] 
		   & &    \ell X'  \ar@2{-}[d]
		   \\
		   &      G 
		   & & &  \ell X' 															   & & &  G 
		   & & &  \ell X', 
		  }
\end{equation}
while the diagram $\Diamond$ is
\begin{equation} \label{grafdiam}
\xymatrix@C=-0.3pc@R=1.5pc
         {  
                 \ell X  \ar@2{-}[d] 
          & & &          \ar@{-}[dl] 
                         \ar@{}[d]|{\eta} 
                         \ar@{-}[dr] 
          & & &  \ell X' \ar@2{-}[d]							
          & & 	 \ell X  \ar@<4pt>@{-}'+<0pt,-6pt>[d] 
                         \ar@<-4pt>@{-}'+<0pt,-6pt>[d]^{R}	
          & &    \ell X' \ar@2{-}[d]
          \\
				 \ell X  \ar@2{-}[d] 
	      & &    \ell X  \ar@2{-}[d] 
	      & &    \ell X  \ar@<4pt>@{-}'+<0pt,-6pt>[d] 
	                     \ar@<-4pt>@{-}'+<0pt,-6pt>[d]^{R}
	      & &    \ell X' \ar@2{-}[d] 	
	      &       \hspace{3ex} = \hspace{3ex}  
	      & 	 \ell X' \ar@{-}[dr] 
	      &              \ar@{}[d]|{\mu_{X'} } 
	      &      \ell X' \ar@{-}[dl]
	      \\
				 \ell X  \ar@{-}[dr] 
		  &              \ar@{}[d]|{\mu_X } 
		  &      \ell X  \ar@{-}[dl] 
		  & &    \ell X' \ar@{-}[dr] 
		  &              \ar@{}[d]|{\varepsilon} 
		  &      \ell X' \ar@{-}[dl] 		
		  & & &   G.
		  \\
		  &       G 
		  & & & & &	
		 }
\end{equation}
Recall that the triangular equations of a duality pairing are:
\begin{equation} \label{triangular}
\xymatrix@C=-0.3pc@R=0.1pc
         { 			
          &  \ar@{-}[ldd] \ar@{-}[rdd] \ar@{}[dd]|{\eta} 
          & & &  X \ar@2{-}[dd] 					
          & & & & & & & & &  Y \ar@2{-}[dd] 
          & & & \ar@{-}[ldd] \ar@{-}[rdd] \ar@{}[dd]|{\eta} 
          & & & &  
          \\
		  & & & & & & & X \ar@2{-}[dd] 
		  & & & & & & & & & & & & &  Y \ar@2{-}[dd] 
		  \\
		    X \ar@2{-}[dd] 
		  & & Y \ar@{-}[rdd] 
		  & \ar@{}[dd]|{\varepsilon} 
		  & X \ar@{-}[ldd] 
		  & \textcolor{white}{X} = \textcolor{white}{X} 
		  & & & & & \textcolor{white}{XX} \hbox{and} \textcolor{white}{XX} 
		  & & & Y \ar@{-}[ddr] 
		  & \ar@{}[dd]|{\varepsilon} 
		  & X \ar@{-}[ldd] 
		  & & Y \ar@2{-}[dd] 
		  & \textcolor{white}{X} = \textcolor{white}{X} 
		  & &  
		  \\
		  & & & & & & & X & & & & & & & & & & & & &  Y. 
		  \\
			X 
		  & & & & & & & & & & & & & & & & & Y	
		 }
\end{equation}

{\it Proof of \eqref{commorph} $\implies$ \eqref{grafdiam}}:
$$
\xymatrix@C=-0.3pc@R=1.5pc
          {
                   \ell X  \ar@2{-}[d] 
           & & &           \ar@{-}[ld] 
                           \ar@{-}[rd] 
                           \ar@{}[d]|{\eta} 
           & & &   \ell X' \ar@2{-}[d] 						
           & & 	   \ell X  \ar@2{-}[d] 
           & & &           \ar@{-}[ld] 
                           \ar@{-}[rd] 
                           \ar@{}[d]|{\eta} 
           & & &   \ell X' \ar@2{-}[d] 		
           & & 	   \ell X  \dr{R} 
           & & &           \ar@{-}[ld] 
                           \ar@{-}[rd] 
                           \ar@{}[d]|{\eta} 
           & & &   \ell X' \ar@2{-}[d] 					
           & & 
           \\
                   \ell X  \ar@2{-}[d] 
           & &     \ell X  \ar@2{-}[d] 
           & &     \ell X  \dr{R} 
           & &     \ell X' \ar@2{-}[d] 							
           & 
           \ \ \ 
                           \ar@{}[d]|= 
           &	   \ell X  \ar@{-}[dr] 
           &               \ar@{}[d]|{\mu_X} 
           &       \ell X  \ar@{-}[dl] 
           & &     \ell X  \dr{R} 
           & &     \ell X' \ar@2{-}[d] 	
           & 
           \ \ \ 
                           \ar@{}[d]|{\stackrel{(\ref{commorph})}{=}} 
           &	   \ell X' \ar@{-}[dr] 
           &               \ar@{}[d]|{\mu_{X'}} 
           &       \ell X' \ar@{-}[dl] 
           & &     \ell X' \ar@2{-}[d] 
           & &     \ell X' \ar@2{-}[d]	
           & 
           \ \ \ 
                   \ar@{}[d]|= 
           & 
           \\
                   \ell X  \ar@{-}[dr] 
           &               \ar@{}[d]|{\mu_X} 
           &       \ell X  \ar@{-}[dl] 
           & &     \ell X' \ar@{-}[dr] 
           &               \ar@{}[d]|{\varepsilon}
           &       \ell X' \ar@{-}[dl] 	
           & & &    G 
           & & &   \ell X' \ar@{-}[dr] 
           &               \ar@{}[d]|{\varepsilon} 
           &       \ell X' \ar@{-}[dl]				
           & & &    G 
           & & &   \ell X' \ar@{-}[dr] 
           &               \ar@{}[d]|{\varepsilon} 
           &       \ell X' \ar@{-}[dl] 						
           & & 
           \\
           &        G 
           & & & & & 																   & & & & & & & & & & & & & & & & & & 
          }
$$
$$\xymatrix@C=-0.3pc @R=1.5pc{
& &   			\ell X \dr{R} & & & \ar@{-}[ld] \ar@{-}[rd] \ar@{}[d]|{\eta} & & & \ell X' \ar@2{-}[d] 					& & 						\ell X \dr{R} & & \ell X' \ar@2{-}[d]					& \ \ \ \ar@{}[dd]|= &	\ell X \dr{R} & & \ell X' \ar@2{-}[d]\\
& \ \ \ \ar@{}[d]|= & 	\ell X' \ar@2{-}[d] & & \ell X' \ar@2{-}[d] & & \ell X' \ar@{-}[dr] & \ar@{}[d]|{\varepsilon} & \ell X' \ar@{-}[dl] 	& \ \ \ \ar@{}[d]|{\stackrel{(\triangle)}{=}} & \ell X' \ar@2{-}[d] & & \ell X' \ar@2{-}[d] 				& & 			\ell X' \ar@{-}[dr] & \ar@{}[d]|{\mu_{X'}} & \ell X' \ar@{-}[dl] \\
& &		 	\ell X' \ar@{-}[dr] & \ar@{}[d]|{\mu_{X'}} & \ell X' \ar@{-}[dl] & & & & 						& & 						\ell X' \ar@{-}[dr] & \ar@{}[d]|{\mu_{X'}} & \ell X' \ar@{-}[dl]	& & 			& G. & \\
& & 			& G & & & & &  													& & 						& G &  									& & 			& & }
$$

{\it Proof of \eqref{grafdiam} $\implies$ \eqref{commorph}}:		
$$\xymatrix@C=-0.2pc @R=1pc{
\X \dr{R} & & & \op{\eta} &	& & 		\X \dd & & & \op{\eta} &		& & 		\X \dd & & & & & & & \op{\eta} &				& & 	\\
\X' & & \X' & & \X' \dd		& \ig{} &	\X \dr{R} & & \X' \dd & & \X' \dd	& \ig{(\ref{grafdiam})} &	\X \dd & & & \op{\eta} & & & \X' \dd & & \X' \dd		& \ig{} & \\
& G \cl{\mu_{X'}}& & & \X'	& & 		\X' & & \X' & & \X' \dd			& & 		\X \dd & & \X \dd & & \X \dr{R} & & \X' \dd & & \X \dd 		& & 	\\
& & & & 			& & 		& G \cl{\mu_{X'}} & & & \X'		& & 		\X & & \X & & \X' & & \X' & & \X' \dd 				& & 	\\
& & & & 			& &		& & & & 				& & 		& G \cl{\mu_X} & & & & \cl{\varepsilon} & & & \X'		& & }$$

$$\xymatrix@C=-0.2pc @R=1pc{
& 	& 	\X \dd & & & \op{\eta} & & & & & 			& & 			\X \dd & & & \op{\eta} & 		& & 		\X \dd & & & \op{\eta} & 	& & \\
& \ig{} & 	\X \dd & & \X \dd & & \X \dr{R} & & & & 		& \ig{(\triangle)} & 	\X \dd & & \X \dd & & \X \dr{R} 	& \ig{} &	\X & & \X & & \X \dr{R}		& & \\
& & 		\X & & \X & & \X' \dd & & & \op{\eta} & 		& & 			\X & & \X & & \X' \dd			& &		& G \cl{\mu_X} & & & \X'.	& & \\
& & 		& G \cl{\mu_X} & & & \X' & & \X' & & \X' \dd 	 	& & 			& G \cl{\mu_X} & & & \X' 		\\
& & 		& & & & & \cl{\varepsilon} & & & \X' }$$

\end{proof}

\section{The Galois and the Tannakian contexts}

{\bf The Galois context.}

\begin{sinnadastandard} {\bf The localic group of automorphisms of a functor.}  \label{Aut(F)}
\end{sinnadastandard}
Let $\cc{E}ns \mr{f} \cc{E}$ be any pointed topos, with inverse image \mbox{$f^* = F$, $\cc{E} \mr{F} \cc{E}ns$.} 
The localic group of automorphisms of $F$ is defined to be the universal  $\rhd$-cone of $\ell$-bijections in the category of locales, as described in the following diagram (see \cite{D1}):
\begin{equation} \label{AutF}
\xymatrix
        {
         FX \times FX \ar[rd]^{\lambda_{X}}  
                      \ar@(r, ul)[rrd]^{\phi_{X}} 
                      \ar[dd]_{F(f) \times F(f)}  
        \\
         {} \ar@{}[r]^(.3){\geq}
         & \;\; Aut(F) \;\; \ar@{-->}[r]^{\phi} 
         & \;H.  
        \\
         FY \times FY \ar[ru]^{\lambda_{Y}} 
         \ar@(r, dl)[rru]^{\phi_{Y}} 
         && {(\phi \; \text{a locale morphism})}
        } 
\end{equation}
From propositions \ref{trianguloesdiamante} and \ref{extension} it immediately follows
\begin{proposition}
The localic group $Aut(F)$ exists and it is isomorphic to the localic group of automorphisms of the restriction of $F$ to any small site of definition for $\cc{E}$. \cqd
\end{proposition}

A point $Aut(F) \mr{\phi} 2$ 
corresponds exactly to the data defining a natural isomorphism of $F$.  
Given $(a,\,b) \in FX \times FX$, 
we will denote \mbox{$\langle X, \,a|b \rangle  = \lambda_X(a,\,b)$.} This element of $Aut(F)$ corresponds to the open set  $\{\phi \,|\, \phi_X(a) = b\}$ of the subbase for the product topology in the set of natural isomorphisms of $F$. For details of the construction of this locale see \cite{D1}.

\vspace{1ex}
 
The $\ell$-bijections $\lambda_X$ determine morphisms of locales 
\mbox{$Aut(FX) \mr{\mu_X} Aut(F)$,} 
$\mu_X \langle a|b \rangle = \langle X, \,a|b \rangle$.
It is straightforward to check that 
%for each $X \in \cc{X}$ 
the following three families of arrows are $\rhd$-cones of  
$\ell$-bijections:
\begin{multline}  \label{groupAut}
FX \times FX \mr{w_X} Aut(F) \otimes Aut(F),\;\; 
w_X(a,\,b) = \displaystyle\bigvee_{x \in FX} 
\langle X, \,a|x \rangle \otimes \langle X, \,x|b \rangle,  
\\
\hspace{-26ex} FX \times FX \mr{\iota_X} Aut(F), \;\;
\iota_X(a,\,b) = \langle X, \,b|a \rangle,
\\
\hspace{-36ex} FX \times FX \mr{e_X} 2, \;\; e_X(a,\,b) = \delta_{a=b}.
\\
\end{multline}
By the universal property they determine localic morphisms with domain $Aut(F)$ which define a localic group structure on $Aut(F)$, such that 
$\mu_X$ becomes an action of $Aut(F)$ on $FX$, and such that for any $X \mr{f} Y \in \cc{E}$, $F(f)$ is a morphism of actions. 
In this way there is a lifting $\tilde{F}$  of the functor $F$ into $\beta^G$,  
$\cc{E} \mr{\tilde{F}} \beta^G$, for $G = Aut(F)$.

 \begin{sinnadastandard} {\bf The (Neutral) Tannakian context associated to pointed topos.} \label{tannakacontext}

For generalities, notation and terminology  concerning Tannaka theory see appendix \ref{appendix}. We consider a  topos with a point $\cc{E}ns \mr{f} \cc{E}$, with inverse image $f^* = F$, $\cc{E} \mr{F} \cc{E}ns$. 
We have a diagram (see \ref{rel}):
$$
\xymatrix
        {
           \cc{E} \ar[r] \ar[d]_F
         & \cc{R}el(\cc{E}) \ar[d]^{\cc{R}el(F)}
         \\
           \cc{E}ns \ar[r] 
         & \cc{R}el
         & \hspace{-7ex} = Sup_0 
        }
$$
This determines a Tannakian context as in appendix \ref{appendix}, with 
$\cc{X} = \cc{R}el(\cc{E})$, $\cc{V} = Sup$, $\cc{V}_0 = \cc{R}el = Sup_0$ and $T = \cc{R}el(F)$. Furthermore, in this case $\cc{X}$, $\Vat$ are symmetric, $T$ is monoidal (\ref{tensoriso}, \ref{rel}),
and every object of $\cc{X}$ has a right dual. Thus, the (large) coend $End^\lor(T)$ (which exists, as we shall see) is a (commutative) Hopf algebra (proposition \ref{hopf}). 
\end{sinnadastandard}

The universal property which defines the coend $End^\lor(T)$ is that of a universal $\Diamond$-cone in the category of sup-lattices, as described in the following diagram: 
$$
\xymatrix
        {
         & TX \times TX \ar[rd]^{\lambda_{X}}  
                      \ar@(r, ul)[rrd]^{\phi_{X}}  
        \\
         TX \times TY \ar[rd]_{TR \times TY} 
			           \ar[ru]^{TX \times TR^{op}} 
	     & \equiv 
	     & \;\;End^\lor(T)\;\; \ar@{-->}[r]^{\phi}  
         & \;H.  
        \\
         & TY \times TY \ar[ru]^{\lambda_{Y}} 
                          \ar@(r, dl)[rru]^{\phi_{Y}} 
         && {(\phi \; \text{a linear map})}
        } 
$$

Given $(a,\,b) \in TX \!\times TX$, 
we will denote \mbox{$[X, \,a,b] = \lambda_X \langle a,\,b \rangle$.} 

From proposition \ref{extension} and \ref{dim1ydim2esdim} it immediately follows:
\begin{proposition} \label{EndhasX}
The large coend defining $End^\lor(T)$ exists and can be computed by the coend corresponding to the restriction of $\, T$ to the full subcategory of $Rel(\Eat)$ whose objects are in any small site $\Cat$ of definition of $\cc{E}$. \cqd
\end{proposition}

By the general Tannaka theory we know that the sup-lattice $End^\lor(T)$ is a Hopf algebra in $Sup$. The description of the multiplication $m$ and a unit $u$ given below proposition \ref{bialg} yields in this case, for $X, \, Y \in \cc{X}$ (here, $F(1_\Cat) = 1_{\Ens} = \{*\}$):
\begin{equation} 
m([X,\, a,a'],\, [Y, \,b,b']) \;= \; [X \times Y,\, (a,\,b),(a',\,b')],\; \;\; u(1) \;=\; [1_\Cat, \, *,*] .
\end{equation} 
This shows that $TX \times TX \mr{\lambda_X} End^\lor(T)$ is a compatible 
$\Diamond$-cone, thus by proposition \ref{compislocale} it follows that $End^\lor(T)$ is a locale,  with top element $[1_\Cat, \, *,*]$ and infimum 
$[X,\, a,a'] \wedge [Y, \,b,b'] =  [X \times Y,\, (a,\,b),(a',\,b')]$.

\vspace{1ex}

We let the reader check the following:
\begin{sinnadaitalica} \label{groupEnd}
The descriptions in the general Tannaka theory of the comultiplication $w$, the counit $\varepsilon$ and the antipode $\iota$ (see appendix \ref{appendix}) yield in this case the formulas 
\mbox{$w_X(a,\,b) = \displaystyle\bigvee_{x \in FX} 
[X, \,a,x] \otimes [X, \,x,b]$,} 
$\iota_X(a,\,b) = [X, \,b,a]$, \mbox{and 
$\varepsilon_X(a,\,b) = \delta_{a=b}$.} \cqd 
\end{sinnadaitalica} 

\vspace{1ex}

\begin{sinnadastandard} 
{\bf The isomorphism $End^\lor(T) \cong Aut(F)$.}
\end{sinnadastandard} 

From propositions  \ref{trianguloesdiamante} and  \ref{Diamondisbijection} it immediately follows (recall that T = F on $\cc{E}$) that $TX \times TX \mr{\lambda_X} Aut(F)$ and 
$TX \times TX \mr{\lambda_X} End^\lor(T)$ are both $\rhd$-cones and $\Diamond$-cones of $\ell$-bijections. From  proposition \ref{supisloc} and the respective universal properties it follows that they are isomorphic locales respecting the cone maps $\lambda_X$. Furthermore, by the formulas in \ref{groupAut} and \ref{groupEnd} we see that under this isomorphism the comultiplication, counit and antipode correspond. Thus, we have: 
\begin{theorem} \label{G=H}
Given any pointed topos, there is a unique isomorphism of localic groups $End^\lor(T) \cong Aut(F)$ commuting with the $\lambda_X$. \cqd  
\end{theorem}

\section{The main Theorems}
A pointed topos $\cc{E}ns \mr{f} \cc{E}$, with inverse image $f^* = F$, $\cc{E} \mr{F} \cc{E}ns$, determines a situation described in the following \mbox{diagram:}
$$
\xymatrix
        {
          \beta^G            \ar[r] \ar[rdd]
        & \cc{R}el(\beta^G)  \ar[r]^{=}
        & Cmd_0(G)          \ar[r]^{=}
        & Cmd_0(H)          \ar[ldd]
        \\
        & \cc{E}             \ar[d]^F \ar[r] \ar[lu]_{\widetilde{F}}
        & \cc{R}el(\cc{E})   \ar[ul]_{\cc{R}el(\widetilde{F})}  
                             \ar[d]^T \ar[ur]^{\widetilde{T}}
        \\
        & \cc{E}ns            \ar[r]
        & \cc{R}el 
        &  \hspace{-8ex} = Sup_0 \subset Sup.
       }
$$

\vspace{1ex}

\noindent
where $G = Aut(F)$, $T = \cc{R}el(F)$, $H = End^\lor(T)$ and
the two isomorphisms in the first row of the diagram are given by Theorems \ref{Comd=Rel} and \ref{G=H}. 

\begin{theorem} \label{AA}
The (Galois) lifting functor $\widetilde{F}$ is an equivalence if and only if the (Tannaka) lifting functor $\widetilde{T}$ is such. \cqd
\end{theorem} 
 Assume now that $\cc{E}$ is a connected atomic topos. The full subcategory of connected objects $\cc{C} \subset \cc{E}$ furnished with the canonical topology is a small site for $\cc{E}$. In \cite{D1} it is proved that the diagram of the functor $F$ restricted to this site $\cc{C} \mr{F} \cc{E}ns$ is a poset (\emph{This fact distinguishes atomic topoi from general locally connected topoi}),   
an explicit construction of $Aut(F)$ is given, and the following key result of localic Galois Theory is proved:
\begin{theorem}[\cite{D1} 6.9, 6.11] \label{keyLGT} ${}$

\vspace{1ex}

1) For any $C \in \cc{C}$ and $(a,\,b) \in FC \times FC$, 
$\; \langle C, \,a|b \rangle \neq 0$. 

\vspace{1ex}

2) Given any other $(a',\,b') \in FC' \times FC'$, if 
$\langle C, \,a|b \rangle \leq \langle C', \,a'|b' \rangle$, then there exists $C \mr{f} C'$ in $\cc{C}$ such that $a' = F(f)(a)$, $b' = F(f)(b)$.
\end{theorem}

\noindent The following theorem follows from \ref{keyLGT} by a formal topos theoretic reasoning. 

\begin{theorem}[\cite{D1} 8.3] \label{fundamentalLGT}
The (Galois) lifting functor $\widetilde{F}$ is an equivalence if and only if the topos $\cc{E}$ is connected atomic.\cqd
\end{theorem}
From \ref{AA} and \ref{fundamentalLGT} we have:
\begin{theorem} \label{BB}
The (Tannaka) lifting functor $\widetilde{T}$ is an equivalence if and only if the topos $\cc{E}$ is connected atomic. \cqd
\end{theorem}

\appendix

\section{Tannaka theory} \label{appendix}

{\bf The Hopf algebra of automorphisms of a $\cc{V}$-functor.} 

(For details see for example \cite{SP}, \cite{S}). Let $\Vat$ be a cocomplete monoidal closed category with tensor product $\otimes$, unit object $I$ and internal hom-functor $hom$. By definition for every object $V \in \cc{V}$,  $hom(V, -)$ is right adjoint to $(-) \otimes V$. That is, for every $X,\;Y$, $hom(X \otimes V,\, Y) = hom(X,\, hom(V,\, Y))$.

A pairing between two objects $V$, $W$ is a pair of arrows $W \otimes V \mr{\varepsilon} I$ and $I \mr{\eta} V \otimes W$ satisfying the usual triangular equations. We say that $W$ is the \emph{left} dual of $V$, and denote $W = V^\vee$, and that $V$ is \emph{right} dual of $W$ and denote $V = W^\wedge$. When $X$ has a left dual, then $X^\vee = hom(X,\, I)$.

The following are basic equations:

If $X$ has a right dual: \hfill 
$Y$ has a left dual $\iff$ 
$hom(Y,\, X)^\wedge = Y \otimes X^\wedge$,

\hfill $X = X^{\wedge^{\scriptstyle\vee}}$,
$hom(X^\wedge,\,Y) = Y \otimes X$.   

\vspace{1ex}

If $X$ has a left dual:  \hfill
$X = X^{\vee^{\scriptstyle\wedge}}$,
$hom(X,\,Y) = Y \otimes X^\vee$.

\vspace{1ex}

Recall that the object of natural transformations between $\cc{V}$-valued functors $L,\,T: \cc{X} \to \cc{V}$, is given, if it exists, by the following end 
\begin{equation}
Nat(L,\,T) = \displaystyle\int_X hom(LX,TX)\,.
\end{equation}

We consider a pair $(\cc{V}_0,\, \cc{V})$, where $\cc{V}_0 \subset \cc{V}$ is a full subcategory such that all its objects have a right dual. 

Let $\cc{X}$ be a $\cc{V}$-category such that for any two functors  
$\cc{X} \mr{L} \cc{V}$ and $\cc{X} \mr{T} \cc{V}_0$ the coend in the following definition exists in $\cc{V}$ (for example, if $\cc{X}$ is small). Then, we define (in Joyal's terminology) the \emph{Nat predual} as follows:
\begin{equation} \label{predual}
Nat^\lor(L,T) = \int^X LX \otimes (TX)^\wedge = 
                                         \int^X hom(LX,\, TX)^\wedge\,.
\end{equation} 

However, the last expression is valid only if $LX$ has a left dual for every $X$ (for example, if $\cc{X} \mr{L} \cc{V}_0$ and every object in 
$\cc{V}_0$ also has a left dual).

\vspace{1ex}

Given $V \in \cc{V}$, recall that there is a functor 
$\cc{X} \mr{V \otimes T} \cc{V}$ defined by \mbox{$(V \otimes T)(X) = V \otimes TX$}, We have:

\begin{proposition} \label{predualprop}
Given $T \in {\Vat_0}^{\cc{X}}$, we have a $\cc{V}$-adjunction $$\xymatrix { {\Vat}^{\cc{X}} \ar@/^/[r]^{Nat^{\lor}(-,T)}_\bot & \Vat \ar@/^/[l]^{(-) \otimes T} }.$$ 
\end{proposition}

\vspace{-3ex}

\begin{proof} ${}$

\noindent
$hom(Nat^\vee(L,\,T), V) 
= hom(\displaystyle\int^X LX \otimes TX^\wedge,\, V) 
= \displaystyle\int_X hom(LX \otimes TX^{\wedge}, \, V)$ 

\noindent \mbox{$ = \displaystyle\int_X hom(LX,\, hom(TX^\wedge, \, V) 
= \displaystyle\int_X hom(LX,\, V \otimes TX) 
= Nat(L,\, V \otimes T)$.}
\end{proof}

In particular we have that the end $Nat(L,\, T)$ exists and \mbox{$Nat(L,\, T) = hom(Nat^\lor(L,\,T), \,I)$.}
It follows that $Nat^\vee(L,\,T)$ classifies natural transformations $L \implies T$ in the sense that they correspond to arrows $Nat^\vee(L,\,T) \mr{} I$ in $\cc{V}$. This does not mean that $Nat(L,\, T)$ is the left dual of $Nat^\vee(L,\,T)$,  
which in general will not have a left dual. 
\emph{Passing from $Nat^\vee(L,\,T)$ to $Nat(L,\, T)$ looses information.}

\vspace{1ex}

The unit of the adjunction $L \Mr{\eta} Nat^\lor (L,\,T) \otimes T$ is a \mbox{coevaluation,} and if $\cc{X} \mr{H} \Vat_0$, it induces (in the usual manner) a \mbox{cocomposition} 
\mbox{$Nat^\lor (L,\,H) \mr{w} Nat^\lor(L,\,T) \otimes Nat^\lor (T,\,H)$.} There is a counit 
$Nat^\lor (T,\,T) \mr{\varepsilon} I$ determined by the arrows $TC \otimes TC^\lor \mr{\varepsilon} I$ of the duality.
All the preceding means exactly that the functors $\cc{X} \mr{} \cc{V}_0$ are the objects of a $\Vat$-cocategory.

\vspace{1ex}

We define $End^\lor(T) = Nat^\lor (T,\,T)$, which is therefore a coalgebra in $\Vat$. The coevaluation in this case becomes a $End^\lor(T)$-comodule structure \mbox{$TC\mr{\eta_C} End^\lor (T) \otimes TC$ on $TC$.} In this way there is a lifting of the functor $T$ into $Cmd_0(H)$, $\cc{X} \mr{\tilde{T}} Cmd_0(H)$, for $H = End^\lor(T)$, and $Cmd_0(H)$ the full subcategory of comodules with underlying object in $\cc{V}_0$.

\vspace{1ex}

\begin{proposition}\label{bialg}
 If $\Xat$ and $T$ are monoidal, and $\Vat$ has a symmetry, then $End^\lor(T)$ is a bialgebra. If in addition $\Xat$ has a symmetry and $T$ respects it, $End^\lor(T)$ is commutative (as an algebra). \cqd
\end{proposition}

We will not prove this proposition here, but show how the multiplication and the unit are constructed, since they are used explicitly in \ref{tannakacontext}.
The multiplication \mbox{$End^\lor(T) \otimes End^\lor(T) \mr{m} End^\lor(T)$} is induced by the composites
$$
m_{X,Y}: TX \otimes TX^\wedge \otimes TY \otimes TY^\wedge \mr{\cong} T(X \otimes Y) \otimes T(X \otimes Y)^\wedge \mr{\lambda_{X \otimes Y}} End^\lor(T).
$$
The unit is given by the composition
$$
u: I \rightarrow I \otimes I^\wedge \mr{\cong} T(I) \otimes T(I)^\wedge \mr{\lambda_{I}} End^\lor(T).
$$
\begin{proposition}\label{hopf}
If in addition to the hypothesis of \ref{bialg} every object of $\Xat$ has a right dual, then $End^\lor(T)$ is a Hopf algebra. \cqd
\end{proposition}

The antipode $End^\lor(T) \mr{\iota} End^\lor(T)$ is induced by the composites
$$
\iota_X: TX \otimes TX^\wedge  \mr{\cong} T(X^\wedge) \otimes TX \mr{\lambda_{X^\wedge}} End^\lor(T).
$$

\section{Elevators calculus} \label{ascensores}

This is a graphic notation\footnote{Invented by the first author in 1969 (which has remained for private draft use for understandable typographical reasons).} 
to write equations in monoidal categories, ignoring associativity and suppressing the tensor  symbol $\otimes$ and the neutral object $I$. Arrows are 
written as cells, the identity arrow as a double line, and the symmetry as crossed double lines. The notation, in particular, exhibits clearly the permutation associated 
to a composite of symmetries, allowing to see if any two composites are the same simply by checking that they codify the same permutation\footnote
         { This is justified by a simple coherence theorem for symmetrical categories (\cite{S} Proposition 2.3), particular case of \cite{JS2} Corollary 2.2 for braided categories.
         }. Compositions are read from top to bottom. 

Given arrows $C \mr{f} D$, $C' \mr{f'} D'\,$: 

The bifunctoriality of the tensor product is the basic equality:
\begin{equation} \label{ascensor}
\xymatrix@C=0ex
         {
             C \dcell{f} & C' \did
          \\
             D \did & C' \dcell{f'}
          \\
             D  &  D' 
         }
\xymatrix@R=6ex{\\ \;\;\;=\;\;\; \\}
\xymatrix@C=0ex
         {
             C \did & C'\dcell{f'}
          \\
             C \dcell{f} & D' \did
          \\
             D & D' 
         }
\xymatrix@R=6ex{ \\ \;\;\;=\;\;\; \\}
\xymatrix@C=0ex@R=0.9ex
         {
             {} & {}
          \\
               C   \ar@<4pt>@{-}'+<0pt,-6pt>[ddd] 
                   \ar@<-4pt>@{-}'+<0pt,-6pt>[ddd]^{f}
             & C'  \ar@<4pt>@{-}'+<0pt,-6pt>[ddd] 
                   \ar@<-4pt>@{-}'+<0pt,-6pt>[ddd]^{f'}
          \\ 
             {} & {}
          \\ 
             {} & {}
          \\
             D & D'.
         }
\end{equation}
This allows to move cells up and down when there are no obstacles, as if they were elevators. 

The naturality of the symmetry is the basic equality:
\begin{equation} \label{swap}
\xymatrix@C=0ex
         {
             C \dcell{f} & C' \did
          \\
             D \did & C' \dcell{f'}
          \\
             D \ar@{=} [dr] & D' \ar@{=} [dl]
          \\
             D' & D 
         }
\xymatrix@R=10ex{ \\ \;\;\;=\;\;\; \\}
\xymatrix@C=0ex
         {
             C \dcell{f} & C' \did
          \\
             D \ar@{=} [dr] & C' \ar@{=} [dl]
          \\
             C' \dcell{f'} & D \did
          \\
             D' & D 
         }
\xymatrix@R=10ex{ \\ \;\;\;=\;\;\; \\}
\xymatrix@C=0ex
         {
             C \ar@{=} [dr] & C' \ar@{=} [dl]
          \\
             C' \did & C \dcell{f}
          \\
             C' \dcell{f'} & D \did
          \\
             D' & D.
         }
\end{equation}
Cells going up or down pass through symmetries by changing the column.  

\vspace{1ex}

\emph{Combining the basic moves \eqref{ascensor} and \eqref{swap} we form configurations of cells that fit valid equations in order to prove new equations.}

The visual aspect of this calculus really helps to find how a given equation can (or cannot) be derived from another ones.

\end{document}